\documentclass[14pt,DIV14,fleqn,parskip=half]{scrartcl}

\makeatletter
\usepackage[utf8]{inputenc}

\usepackage[backend=biber,style=alphabetic,minalphanames=4,maxalphanames=5,maxnames=5,minnames=3,maxbibnames=99,isbn=false,url=false,doi=false,autocite=inline,labelalpha=true]{biblatex}
\renewbibmacro{in:}{}

\usepackage{microtype}

\usepackage[shortlabels]{enumitem}
\newcommand\numberthis{\addtocounter{equation}{1}\tag{\theequation}}

\usepackage{fourier}
\usepackage{berasans}
\usepackage[scaled=0.85]{beramono}
\usepackage[T1]{fontenc}
\usepackage[cal=pxtx,calscaled=.96]{mathalfa}

\usepackage{tikz} 
\usetikzlibrary{calc}

\usepackage{amsmath} 
\usepackage{mathtools} 
\usepackage{commath} 
\usepackage{amsfonts} 
\usepackage{amssymb} 

\usepackage[amsmath,thmmarks,framed]{ntheorem}
\newtheorem{theorem}{Theorem}
\newtheorem{lemma}[theorem]{Lemma}
\newtheorem{coro}[theorem]{Corollary}
\numberwithin{theorem}{section} 
\numberwithin{equation}{section}
\theorembodyfont{\upshape}
\newtheorem{defi}[theorem]{Definition}
\newtheorem{remark}[theorem]{Remark}
\newtheorem{example}[theorem]{Example}
\theoremstyle{nonumberplain}
\newtheorem{examples}[theorem]{Examples}
\newtheorem{acknow}[theorem]{Acknowledgements}

\theorembodyfont{\upshape}
\theoremheaderfont{\upshape\bfseries}
\theoremseparator{.}
\theoremsymbol{{\ensuremath{\square}}}
\theoremstyle{nonumberplain}
\newtheorem{proof}{Proof}

\newtheoremstyle{proofstyle}%
  {\item[\theorem@headerfont\hskip\labelsep ##1\theorem@separator]}%
  {\item[\theorem@headerfont\hskip\labelsep ##1\  ##3\theorem@separator]}
\theoremstyle{proofstyle}
\theoremsymbol{{\ensuremath{\square}}}
\newtheorem{proofof}{Proof}

\newcommand{\R}{\ensuremath{\mathbb R}}
\newcommand{\N}{\ensuremath{\mathbb N}}
\newcommand{\nc}{\ensuremath{_{NC}}}
\newcommand{\crr}{\ensuremath{_{CR}}}
\newcommand{\poh}{\ensuremath{^{1/2}}}
\newcommand{\pmoh}{\ensuremath{^{-1/2}}}
\newcommand{\inv}{\ensuremath{^{-1}}}
\newcommand{\dx}{\,\dif x}
\newcommand{\ds}{\,\dif s}
\newcommand{\distSq}[2]{\ensuremath{\operatorname{d}^2(#1,#2)}}
\newcommand{\distAlpha}[1]{\ensuremath{\operatorname{d}^2_{#1}}}
\newcommand{\ins}{\ensuremath{\subseteq}}
\newcommand{\fa}{ \ensuremath{\,\forall}}

\newcommand{\cp}{\ensuremath{c_\textup{P}}}

\newcommand{\cjc}{\ensuremath{c_\textup{apx}}}
\newcommand{\ctr}{\ensuremath{c_\textup{tr}}}
\newcommand{\cqi}{\ensuremath{c_\textup{QI}}}
\newcommand{\csr}{\ensuremath{c_\textup{sr}}}
\newcommand{\cquot}{\ensuremath{c_\textup{quot}}}
\newcommand{\cinner}{\ensuremath{M_\textup{int}}}
\newcommand{\cbd}{\ensuremath{M_\textup{bd}}}
\newcommand{\cpatch}{\ensuremath{M_\textup{patch}}}
\newcommand{\cdf}{\ensuremath{c_\textup{dF}}}
\newcommand{\cf}{\ensuremath{c_\textup{F}(\Om)}}
\newcommand{\cinv}{\ensuremath{c_\textup{inv}}}
\newcommand{\com}{\ensuremath{c_\omega}}
\newcommand{\calph}{\ensuremath{c(\Tcal)}}
\newcommand{\csecond}{\ensuremath{C_{2}}}
\newcommand{\joo}{\ensuremath{j_{1,1}}}
\newcommand{\hmax}{\ensuremath{h_\textup{max}}}
\newcommand{\omz}{\ensuremath{\omega_{0}}}

\DeclareMathOperator{\bisec}{bisec}
\DeclareMathOperator{\red}{red}
\DeclareMathOperator{\ddiv}{div}

\newcommand{\nablanc}{\ensuremath{\nabla\nc}}
\DeclareMathOperator{\conv}{conv}
\newcommand{\dell}{\ensuremath{\partial}}
\newcommand{\Inc}{\ensuremath{I\nc}}
\DeclareMathOperator{\diam}{diam}
\DeclareMathOperator{\dist}{dist}
\DeclareMathOperator{\Mid}{mid}
\DeclareMathOperator{\D}{D}
\DeclareMathOperator{\Curl}{Curl}

\newcommand{\Om}{\ensuremath{\Omega}}
\newcommand{\bOm}{\ensuremath{\partial\Omega}}
\newcommand{\Ncal}{\ensuremath{\mathcal N}}
\newcommand{\Tcal}{\ensuremath{\mathcal T}}
\newcommand{\Tcalf}{\ensuremath{ \hat{\mathcal T}}}
\newcommand{\Ecal}{\ensuremath{\mathcal E}}
\newcommand{\Ecalf}{\ensuremath{\hat{\mathcal E}}}
\newcommand{\Admis}{\ensuremath{\mathbb{T}}}
\newcommand{\hk}{\ensuremath{h_K}}
\newcommand{\htc}{\ensuremath{h_\Tcal}}
\newcommand{\jump}[1]{\ensuremath{[#1]}}

\newcommand{\Hnc}{\ensuremath{H^1\nc(\Tcal)}}
\newcommand{\CR}{\ensuremath{CR^1_0(\Tcal)}}
\newcommand{\CRf}{\ensuremath{CR^1_0(\Tcalf)}}
\newcommand{\Soz}{\ensuremath{S^1_0(\Tcal)}}
\newcommand{\So}{\ensuremath{S^1(\Tcal)}}
\newcommand{\Sof}{\ensuremath{S^1(\Tcalf)}}
\newcommand{\Sozf}{\ensuremath{S^1_0(\Tcalf)}}
\newcommand{\HTriang}{\ensuremath{H^1(\Tcal)}}
\newcommand{\Ltwo}[1]{\ensuremath{L^2(#1)}}

\renewcommand\abs[1]{ | #1 |  }
\renewcommand\norm[1]{| \! | #1 | \! |}
\newcommand\NormLtwo[2]{\norm{#1}_{L^2(#2)}}
\newcommand\NormLtwoOm[1]{\norm{#1}_{L^2(\Om)}}
\newcommand\NormLtwoK[1]{\norm{#1}_{L^2(K)}}
\newcommand\NormLtwoT[1]{\norm{#1}_{L^2(T)}}
\newcommand\NormEnergy[1]{| \! | \! | #1 | \! | \! |}
\newcommand\NormEnergyOn[2]{\NormEnergy{ #1 }_{#2}}
\newcommand\NormEnergync[2]{\NormEnergy{ #1 }_{NC(#2)}}
\newcommand\NormEnergyncOnly[1]{\NormEnergy{ #1 }_{NC}}
\newcommand\NormEnergyncK[1]{\NormEnergy{ #1 }_{NC(K)}}
\newcommand\NormEnergyncT[1]{\NormEnergy{ #1 }_{NC(T)}}
 
\newcommand{\vc}{\ensuremath{v_{\text{C}}}}
\newcommand{\vcr}{\ensuremath{v\crr}}
\newcommand{\Jc}{\ensuremath{J_{\text{C}}}}

\def\Xint#1{\mathchoice
{\XXint\displaystyle\textstyle{#1}}%
{\XXint\textstyle\scriptstyle{#1}}%
{\XXint\scriptstyle\scriptscriptstyle{#1}}%
{\XXint\scriptscriptstyle\scriptscriptstyle{#1}}%
\!\int}
\def\XXint#1#2#3{{\setbox0=\hbox{$#1{#2#3}{\int}$}
\vcenter{\hbox{$#2#3$}}\kern-.5\wd0}}
\newcommand{\intmean}{\Xint-}

\newcommand{\ol}[1]{\overline{#1}}

\newlength{\raisebulletlen}
\setbox1=\hbox{$\bullet$}\setbox2=\hbox{\tiny$\bullet$}
\newcommand\pbullet{\raisebox{\raisebulletlen}{\,\scriptsize$\bullet$}\,}
\makeatother

\title{Constants in Discrete Poincar\'e and Friedrichs Inequalities and Discrete Quasi-Interpolation}
\author{Carsten Carstensen \and Friederike Hellwig}
\date{}
\bibliography{lit}

\begin{document}
\maketitle

\begin{abstract}
\textbf{Abstract.} { This paper provides a discrete Poincar\'e inequality in $n$ space
dimensions on a simplex $K$ with explicit constants.  This inequality bounds the norm of the piecewise derivative of
functions with  integral mean zero on $K$ and all integrals of jumps zero
along all interior sides by its Lebesgue norm by $C(n)\diam(K)$.
The explicit constant $C(n)$ depends only on the dimension $n=2,3$
in case of an  adaptive triangulation with the newest vertex bisection.
The second part of this paper proves the stability of an enrichment
operator, which leads to the stability and approximation of a (discrete)
quasi-interpolator applied in the proofs of the discrete Friedrichs
inequality and discrete reliability estimate with explicit bounds on the
constants in terms of the minimal angle \(\omz\) in the triangulation.
The analysis allows the bound of two constants \(\Lambda_1\) and
\(\Lambda_3\) in the axioms of adaptivity for the practical choice of the
bulk parameter with guaranteed optimal convergence rates.


\textbf{AMS Subject Classication. } 65N30
}
\end{abstract}

\section{Introduction}

The first topic is the {\em discrete Poincar\'e inequality} on a simplex $K$ with diameter $h_K$ and a  refinement $\mathcal{T}$
by newest-vertex bisection (NVB) of $K$. Then any compatible piecewise Sobolev function $v\nc$ 
such as  Crouzeix-Raviart functions 
with integral mean zero over $K$ and the piecewise gradient $\nablanc v\nc$ satisfies
\begin{equation}\label{eqccintro1}
\NormLtwoK{ v\nc } \leq C(n) h_K \NormLtwoK{ \nablanc v\nc}
\end{equation}
with a universal constant $C(n)$, which exclusively depends on the dimension $n$. This paper provides bounds 
of $C(n)$ for any dimension $n$ in terms of the refinements from \cite{stevenson08,gss2014}
with $C(2)\le \sqrt{3/8}$ or $C(3)\le\sqrt{5}/3$ and utilizes them to prove an explicit constant in an interpolation error estimate for a discrete nonconforming interpolation operator. The discrete Poincare inequality \eqref{eqccintro1} is utilized e.g. in \cite{HellaSafemI, HellaCarstenElast} without further specification of the discrete Poincare constant. 

The second topic is an enrichment operator $J_1:  \CR\to \Soz$ between the nonconforming
and conforming $P_1$ finite element spaces with respect to a regular triangulation $\mathcal{T}$ into triangles 
for $n=2$ with local mesh-size 
$h_\mathcal{T}$ (defined by $ h_\mathcal{T}|_K=h_K= \diam(K)$ on $K\in \mathcal{T}$) and the approximation property 
\begin{equation}\label{eqccintro2}
\NormLtwoOm{ h_\mathcal{T}^{-1}(\vcr- J_1 \vcr) } \le \cjc  \NormLtwoOm{ \nablanc \vcr } \quad\text{ for all }\vcr\in CR_0^1(\mathcal{T})
\end{equation}
and some global constant $\cjc\le C(\mathcal{T})\sqrt{ \cot(\omz)}$ 
for the minimal angle $\omz$ in the triangulation and some topological constant $C(\mathcal{T})$ which depends 
only on the number of triangles that share one vertex in $\mathcal{T}$. The combination   of 
\eqref{eqccintro2} with an inverse estimate implies stability  of $J_1$
with respect to the piecewise $H^1$ norms. 

Another application of \eqref{eqccintro2} is the {\em discrete Friedrichs inequality } for Crouzeix-Raviart functions 
\begin{equation}\label{eqccintro3}
\NormLtwoOm{\vcr } \le \cdf  \NormLtwoOm{\nablanc \vcr }\quad\text{ for all }\vcr\in CR_0^1(\mathcal{T})
\end{equation}
and some global constant $\cdf$.

The third topic is the quasi-interpolation $J:=J_1\circ \Inc:H^1_0(\Omega)\rightarrow \Soz$, which combines the nonconforming interpolation operator
$\Inc$ with the enrichment operator $J_1$, and guarantees the error estimate
\begin{equation*}
 \NormLtwoOm{\htc\inv(\operatorname{id}-J)v} \leq \cqi \NormEnergy{v}\text{ for all }v\in H^1_0(\Omega)
\end{equation*}
for some global constant \(\cqi\).
This first-order approximation property with \(\cqi\) and some stability constants are derived explicitly in terms of $\cjc$. A special case of this operator yields a discrete quasi-interpolation $J_{dQI}:\Sozf\rightarrow\Soz$ for a triangulation \(\Tcal\) with refinement \(\Tcalf\) such that any $\hat v_C\in\Sozf$ satisfies $\hat v_C= J_{dQI} \hat v_C$ on unrefined elements \(\Tcal\cap\Tcalf\).
This enables applications to the discrete reliability e.g. in \cite{cgs2013b} and generally in the axioms of adaptivity \cite{axioms,safem}
and leads to constants, which allow for a lower bound of the bulk parameter in adaptive mesh refining algorithms for guaranteed optimal convergence rates.

The remaining parts of this paper are organized as follows. The necessary notation on the triangulation and
its refinements follows in Section~\ref{sec:notation} with a discrete trace identity. The discrete Poincare inequality \eqref{eqccintro1} is established  in Section~\ref{sec:discretePoincare}.  The analysis provides an easy proof of the Poincare constant in \(2\)D for a triangle with constant  $1/\sqrt{6}$ which is not too large in comparison with the value $1/\joo$ from \cite{laugesensiudeja}
for the first positive root $\joo$ of the Bessel function of the first kind.  Section~\ref{sec:enrichment} introduces and analyses the  enrichment operator $J_1$ with bounds on
$\cjc$  in \eqref{eqccintro2} and $\cdf$ in \eqref{eqccintro3}. The quasi-interpolation follows in Section~\ref{sec:quasiinterpolation} and the 
application to discrete reliability in Section~\ref{sec:axioms} concludes this paper.  

The analysis of explicit constants is performed in \(2\)D for its clear geometry of a nodal patch with an easy topology.
The \(3\)D analog is rather more complicated as there is no one-dimensional enumeration of all simplices, which 
share one vertex in a triangulation. The results are valid for higher dimension as well but the constants are less 
immediate to derive. The work originated from lectures on computational PDEs at the  Humboldt-Universit\"at zu Berlin over the last years to introduce students to the discrete functions spaces without a deeper introduction of Sobolev spaces. 

\section{Notation}\label{sec:notation}
For $n=2,3$ and any bounded Lipschitz domain $\Omega\ins\R^n$ with polyhedral boundary, let $\Tcal$ denote a regular triangulation of $\Omega$ into $n$-simplices. Let $\Ecal$ (resp. $\Ecal(\Omega)$ or $\Ecal(\bOm)$) denote the set of all sides (resp. interior sides or boundary sides) in the triangulation and $\Ncal$ (resp. $\Ncal(\Omega)$ or $\Ncal(\bOm)$) denote the set of all nodes (resp. interior nodes or boundary nodes) in the triangulation. For any $n$-simplex $T\in\Tcal$ with volume $\abs{T}$, let $\Ecal(T)$ denote the set of its sides (edges for $n=2$ resp. faces for $n=3$), $\Ncal(T)$ the set of its nodes, and let $h_T:=\diam(T)$ be its diameter. For any $L^2$ function $v\in\Ltwo{\omega}$, define the integral mean $\intmean_\omega v \dx := \abs{\omega}\inv \int_\omega v\dx$ for $\omega=T\in\Tcal$ or $\omega=E\in\Ecal$ with surface measure $\abs{E}$.
For any node $z\in\Ncal$, let $\Tcal(z):=\{T\in\Tcal\,|\, z\in\Ncal(T)\}$ and $\omega_z:=\bigcup_{T\in\Tcal(z)} T$ the nodal patch. For $E\in\Ecal$, let $\omega_E = \bigcup_{T\in\Tcal, E\in\Ecal(T)} T$. For $T\in\Tcal$, let $\omega_T:=\bigcup_{z\in\Ncal(T)}\omega_z$ and let $\measuredangle(T,z)$ denote the interior angle of $T$ at the node $z\in\Ncal(T)$.

The unit normal vector $\nu_T$ along $\dell T$ points outward. For any side $E=\partial T_+\cap\partial T_-\in \Ecal$ shared by two simplices, the enumeration of the neighbouring simplices $T_\pm$ is fixed. Given any function $v$, define the jump of $v$ across an inner side $E\in\Ecal(\Om)$ by $\jump{v}_E:=v|_{T_+} - v|_{T_-}\in L^2(E)$ and the jump across a boundary side $E\in\Ecal(\bOm)$ by $\jump{v}_E:= v$.

\begin{defi}[bisection]
 Any $n$-simplex $T=\conv\{P_1,\dots,P_{n+1}\}$ is identified with the $(n+1)$-tuple $(P_1,\dots,P_{n+1})$. Its refinement edge is $\ol{P_1P_{n+1}}$ and $\bisec(T):=\{T_1, T_2\}$ is defined with $T_1:=\conv\{P_1,(P_1+P_{n+1})/2, P_2,\dots,P_n\}$ and\\ $T_2:=\conv\{P_{n+1},(P_1+P_{n+1})/2,P_2,\dots,P_n\}$.
The ordering of the nodes in the $(n+1)$-tuples and thus, the refinement edges, for the new simplices $T_1$ and $T_2$ are fixed and for $n=3$ additionally depend on the type of the (tagged) $n$-simplex \cite{stevenson08}.
\end{defi}

\begin{remark}\label{rem:bisectDiam}
 There exists $M=M(n)\in\N$ such that any $n$-simplex $K$ and $\Tcal:=\bisec^{(M)}(\{K\}):= \bisec(\bisec(\dots(\bisec(\{K\}))\dots))$ satisfies
\begin{equation*}
 \max\{h_T\,|\, T\in\Tcal\} \leq h_K/2.
\end{equation*}
It holds that $M(2)=3$ and $M(3)=7$. (The latter follows from mesh-refining of the reference tetrahedron of all types \cite{stevenson08} by undisplayed computer simulation.)
\end{remark}

\begin{defi}
 Given any initial triangulation $\Tcal_0$, let $\Admis=\Admis(\Tcal_0)$ be the set of all regular triangulations obtained from $\Tcal_0$ with a finite number of successive bisections of appropriate simplices. For any $\Tcal\in\Admis$ and $\omega\subseteq \Omega$, let $\Tcal(\omega):=\{K\in\Tcal\,|\, K\ins \overline{\omega}\}$. Let \(\bigcup\Admis\) the set of all admissible simplices \(T\) with \(T\in\Tcal\) for some \(\Tcal\in\Admis\).
The \emph{level} of an $n$-simplex $T\in\bigcup\Admis$ with $T\ins K\in\Tcal_0$ is defined as $\ell(T):=\log_2(\abs{K}/\abs{T})\in\N_0$.
\end{defi}

\begin{remark}
 For any $\Tcal\in\Admis$ and $T\in\bigcup\Admis$ (not necessarily $T\in\Tcal$), $\Tcal(T)$ satisfies exactly one of the following statements.
\begin{enumerate}[label=(\alph*)]
 \item There exists $K\in\Tcal$ such that $T\ins K$.
 \item $\Tcal(T)\in \Admis(\{T\})$, in particular, $\Tcal(T)$ is a regular triangulation of $T$ with $2\leq \abs{\Tcal(T)}$.
\end{enumerate}
\end{remark}

\begin{defi}
Define the spaces
\begin{align*}
 \HTriang &:= \{ v\in \Ltwo{\Omega} \,|\, \fa T\in\Tcal,\, v|_T\in H^1(\operatorname{int}(T))\equiv H^1(T) \},\\
 \Hnc &:= \{ v\nc\in \HTriang \,|\, \fa E\in\Ecal(\Omega),\, \intmean_E \jump{v\nc}_E \ds = 0 \}.
\end{align*}
Define the discrete spaces
\begin{align*}
 P_1(\Tcal)&:=\{ v_1\in  \Ltwo{\Omega}\,|\, \fa T\in\Tcal,\, v_1|_T\text{ is polynomial of degree }\leq 1\text{ on }T\}\subseteq \HTriang,\\
 \Soz &:= P_1(\Tcal) \cap H^1_0(\Om)\subseteq H^1_0(\Om),\\
 \CR &:= \{ \vcr \in  P_1(\Tcal)\,|\, \fa E\in\Ecal(\Omega),\, \vcr \text{ continous at } \Mid(E),\\
 &\quad\quad\quad\quad\quad\quad\quad\quad\fa E\in\Ecal(\bOm),\, \vcr(\Mid(E))=0\}\subseteq \Hnc.
\end{align*}

For $v\in\HTriang$ let $\nablanc v$ denote the piecewise weak gradient and for any measurable subset $\omega\ins \Omega$, let $\NormEnergync{v}{\omega}:=\NormLtwo{\nablanc v}{\omega}$, $\NormEnergyncOnly{v}:=\NormEnergync{v}{\Omega}$ the nonconforming energy norm.
\end{defi}
A piecewise application of the Gau\ss{} divergence theorem leads to the following discrete trace identity.
\begin{lemma}[Discrete trace identity]\label{lem:discreteTraceIdentity}
 Let $T=\conv\{E,P\}$ be an $n$-simplex with vertex $P\in\Ncal(T)$ and opposite side $E\in\Ecal(T)$ and $\Tcal$ a regular triangulation of $T$. Then any $v\nc\in\Hnc$ satisfies the trace identity
\begin{equation*}
 \intmean_E v\nc\ds = \intmean_T v\nc \dx + \frac1n \,\intmean_T (x-P)\cdot \nablanc v\nc \dx.
\end{equation*}
\end{lemma}

\begin{proof}
The proof is a generalization of the continuous trace identity \cite{CGedickeRim}. Let $\Ecal(\operatorname{int}(T))$ the interior sides with respect to the triangulation $\Tcal$. The identity $\ddiv\nc((\pbullet-P) v\nc) = n v\nc + (\pbullet-P)\cdot \nablanc v\nc$, where \((\pbullet-P)(x)=(x-P)\) for \(x\in T\), a piecewise application of the Gau\ss{} divergence theorem, and the definition of the normal jumps \(\jump{ v\nc}_F \cdot \nu_F = v\nc|_{T_+} \nu_{T_+} + v\nc|_{T_-} \nu_{T_-}\) for \(F=\partial T_+\cap \partial T_-\), \(T_\pm\in\Tcal\), lead to 
\begin{align*}
 n \int_T v\nc\dx  &+ \int_T (x-P)\cdot \nablanc v\nc \dx = \sum_{F\in\Ecal(\operatorname{int}(T))} \int_F \jump{ v\nc}_F (x-P)\cdot \nu_F \ds\\
 &+ \sum_{F\in\Ecal(T)\setminus\{E\}} \int_F  v\nc (x-P)\cdot \nu_F \ds
+ \int_{E} v\nc (x-P)\cdot \nu_{E} \ds.
\end{align*}
The observation of $(x-P)\cdot \nu_F \equiv c_F\in \R$ on any $F\in \Ecal(\operatorname{int}(T))$, $(x-P)\cdot \nu_{F}\equiv 0$ on $F\in\Ecal(T)\setminus\{E\}$ and $(x-P)\cdot \nu_{E} = \dist(P,E)=n\abs{T}/\abs{E}$ on $E$ conclude the proof.
\end{proof}

\begin{lemma}\label{lem:normIdMinusP}
 Any $n$-simplex $T$ with vertex $P\in \Ncal(T)$ and the identity mapping $\pbullet$ (i.e. \((\pbullet-P)(x)=x-P\) for \(x\in T\)) satisfy
\begin{equation*}
 \NormLtwo{\pbullet - P}{T} \leq \sqrt{\frac{n}{n+2}} h_T \abs{T}\poh.
\end{equation*}
\end{lemma}

\begin{proof}
 Let $\lambda_1,\dots,\lambda_{n+1}\in P_1(T)$ be the barycentric coordinates of the $n$-simplex $T=\conv(P_1,\dots,P_{n+1})$. Without loss of generality, assume $P=P_{n+1}=0$. The identity $x=\sum_{j=1}^{n+1} \lambda_j(x) P_j$ implies
\begin{align*}
 \NormLtwo{\pbullet - P}{T}^2 &= \NormLtwo{\sum_{j=1}^{n} \lambda_j P_j}{T}^2 = \sum_{j,k=1}^{n} P_j\cdot P_k \int_T \lambda_j \lambda_k \dx\\
 &=  (\sum_{j,k=1}^{n} P_j\cdot P_k + \sum_{j=1}^{n} \abs{P_j}^2 )  \abs{T}/((n+1)(n+2))
\end{align*}
with the integration formula for the barycentric coordinates $\int_T \lambda_j \lambda_k \dx = \abs{T} (1+\delta_{jk})/((n+1)(n+2))$. The Cauchy inequality and $\abs{P_j}\leq h_T$ lead to the assertion.
\end{proof}

\section{Discrete Poincar\'e Inequality}\label{sec:discretePoincare}
This section establishes a discrete Poincar\'e inequality on an $n$-simplex $K\subseteq\R^n$ with a constant $C(n)=((4M(n)-3)/(3n(n+2)))\poh$ with $M(n)$ from Remark \ref{rem:bisectDiam} and so $C(2)=\sqrt{3/8}$ and $C(3)=\sqrt{5}/3$.
\begin{theorem}[Discrete Poincar\'e inequality]\label{thm:discPoin}
 Let $K$ be an $n$-simplex and $\Tcal\in \Admis(\{K\})$ be a regular triangulation of $K$. Then any $v\nc\in\Hnc$ satisfies
\begin{equation*}
 \NormLtwoK{v\nc - \intmean_K v\nc\dx} \leq C(n) \hk \NormEnergyncK{v\nc}.
\end{equation*}
\end{theorem}

The proof of this theorem utilizes a distance function
\begin{equation*}
 \distSq{f}{T} := \NormLtwo{f-\intmean_T f\dx}{T}^2
\end{equation*}
and its behavior under bisection for any $f\in \Ltwo{T}$ in an $n$-simplex $T\ins \R^n$.

\begin{lemma}\label{lem:estimateDistBisec}
 Let $\,\Tcal\in\Admis(\{K\})$, $T\in\cup\Admis(\{K\})$, and $\{T_1,T_2\}=\bisec(T)$. Then any $v\nc\in\Hnc$ satisfies
\begin{equation*}
 \distSq{v\nc}{T} \leq (n(n+2))\inv\max_{j=1,2}{h_{T_j}^2} \NormEnergyncT{v\nc}^2 + \sum_{j=1,2} \distSq{v\nc}{T_j}.
\end{equation*}

\end{lemma}
\begin{proof}
Let $F:= \dell T_1 \cap \dell T_2$ and $P_1,P_2\in\Ncal(T)$ with $T_j=\conv\{F,P_j\}$ for $j=1,2$. Since $T\in\cup\Admis(\{K\})$ and $\Tcal\in\Admis(\{K\})$, it holds either $T\ins \hat T \in \Tcal$ for some $\hat T \in \Tcal$ or $\Tcal(T)$ is a regular triangulation of $T$. Hence, $v\nc\in\Hnc$ implies $\int_F \jump{v\nc}_F \ds = 0$ in both cases and $v_F:= \intmean_F v\nc \ds$ is well-defined.
Similarly, for $j=1,2$, either $T_j\ins \hat T_j \in \Tcal$ for some $\hat T_j \in \Tcal$ or $\Tcal(T_j)$ is a regular triangulation of $T_j$. Therefore, $v\nc|_{T_j}\in H^1(T_j)$ or $v\nc|_{T_j} \in H^1\nc(\Tcal(T_j))$ and thus, Lemma \ref{lem:discreteTraceIdentity} is applicable on $T_1$ and $T_2$. With $\ol{v}_j:= \intmean_{T_j} v\nc \dx$ for $j=1,2$, the Cauchy Schwarz inequality and Lemma \ref{lem:normIdMinusP} imply
\begin{align*}
 n\abs{\ol{v}_j - v_F} &=  \Big\lvert \intmean_{T_j} (x-P_j)\cdot \nablanc v\nc\dx \Big\rvert\\
&\leq \NormLtwo{\pbullet - P_j}{T_j} \NormEnergync{v\nc}{T_j} /\abs{T_j}\\
&\leq \frac{\sqrt{n} h_{T_j}}{\sqrt{(n+2)\abs{T_j}}}\NormEnergync{v\nc}{T_j}.\numberthis\label{eqn:vjminusvf}
\end{align*}
With $\ol{v}:=\intmean_T v\nc \dx = (\ol{v}_1+\ol{v}_2)/2$, the triangle inequality yields
\begin{equation*}
 \sum_{j=1,2} \abs{\ol{v}-\ol{v}_j}^2 = \abs{\ol{v}_1 - \ol{v}_2}^2/2 \leq \sum_{j=1,2} \abs{\ol{v}_j - v_F}^2.
\end{equation*}
This, the orthogonality of $v\nc-\ol{v}_j$ onto $\ol{v}-\ol{v}_j$ in $\Ltwo{T_j}$, and $\abs{T_1}=\abs{T_2}$ show
\begin{align*}
 \distSq{v\nc}{T} &= \NormLtwo{v\nc -  \ol{v}}{T_1}^2 + \NormLtwo{v\nc -  \ol{v}}{T_2}^2\\
&= \sum_{j=1,2} \big( \NormLtwo{v\nc -  \ol{v}_j}{T_j}^2 + \NormLtwo{\ol{v} -  \ol{v}_j}{T_j}^2\big)\\
&\leq \sum_{j=1,2} \big( \NormLtwo{v\nc -  \ol{v}_j}{T_j}^2 + \abs{T_j} \abs{\ol{v}_j- v_F}^2\big).
\end{align*}
The combination with \eqref{eqn:vjminusvf} concludes the proof.
\end{proof}

\begin{proofof}[of Theorem \ref{thm:discPoin}]
Let $\Tcal_0 := \{K\}$ and $\Tcal_\ell := \bisec^{(\ell)}(\Tcal_0)\in \Admis(\Tcal_0)$ for any $\ell\in\N_0$. For any multiindex $\alpha=(\alpha_1,\dots,\alpha_\ell)\in \{1,2\}^\ell$ of length $\dim\alpha=\ell\in\N_0$, define the $n$-simplex $K_\alpha$ recursively by $K_\emptyset := K$ and $\{K_{(\alpha,1)},K_{(\alpha,2)}\} = \bisec(K_\alpha)$ for extended multiindices $(\alpha,1)$ and $(\alpha,2)$ in $\{1,2\}^{\ell+1}$. This implies $\Tcal_\ell =\{K_\alpha\,|\, \dim\alpha=\ell\}$ and $h_\ell:= \max_{T\in\Tcal_\ell} h_T$ satisfies $h_{\ell+1}\leq h_\ell$. Remark \ref{rem:bisectDiam} shows that any $\ell\in\N_0$ satisfies $h_{\ell+M}\leq h_\ell/2$ for fixed $M=M(n)\in \N$, thus $h_{kM}\leq h_0 2^{-k}$ for $k\in \N_0$. This implies
\begin{equation}\label{eqn:geometricSeriesH}
 \sum_{\ell=0}^{\infty} h_\ell^2 = \sum_{k=0}^\infty \sum_{\ell=kM}^{(k+1)M-1} h_\ell^2 \leq M \sum_{k=0}^\infty h_{kM}^2 \leq M \sum_{k=0}^\infty h_{0}^2 2^{-2k} = 4 M h_{0}^2 /3.
\end{equation}
With $\distAlpha{\alpha}:=\distSq{v\nc}{K_\alpha}$ for any $\ell\in\N_0$ and any $\alpha\in\{1,2\}^\ell$, Lemma \ref{lem:estimateDistBisec} and the abbreviation $\gamma := (n(n+2))\inv$ show
\begin{equation*}
 \distAlpha{\alpha} \leq \gamma h^2_{\dim\alpha+1} \NormEnergync{v\nc}{K_\alpha}^2 + \sum_{j=1,2} \distAlpha{(\alpha,j)}.
\end{equation*}
The sum over all multiindices of length $k\in\N_0$ reads
\begin{equation*}
 \sum_{\alpha\in\{1,2\}^k} \distAlpha{\alpha} \leq \gamma h^2_{k+1} \NormEnergync{v\nc}{K}^2 + \sum_{\beta\in\{1,2\}^{k+1}} \distAlpha{\beta}.
\end{equation*}
Successive applications of this result and any choice of $L\geq \max_{T\in\Tcal} \ell(T)$ lead to
\begin{align*}
 \distAlpha{\emptyset} &\leq \gamma h^2_{1} \NormEnergync{v\nc}{K}^2 + \sum_{\ell =1,2} \distAlpha{\ell}\\
&\leq \gamma (h^2_{1}+ h^2_2) \NormEnergync{v\nc}{K}^2 + \sum_{\alpha\in\{1,2\}^2} \distAlpha{\alpha}\\
&\leq \dots \leq \gamma \big(\sum_{\ell=1}^L h_\ell^2\big) \NormEnergync{v\nc}{K}^2 + \sum_{\alpha\in\{1,2\}^L} \distAlpha{\alpha}.\numberthis\label{eqn:succEstimate}
\end{align*}
Since $L\geq \max_{T\in\Tcal} \ell(T)$, $\Tcal_L=\bisec^{(L)}(K)$ is finer than $\Tcal$. Therefore $v\nc|_{K_\alpha} \in H^1(K_\alpha)$ for $\dim\alpha\geq L$ and the Poincar\'e inequality shows
\begin{equation*}
 \distAlpha{\alpha} = \NormLtwo{v\nc - \intmean_{K_\alpha} v\nc\dx}{K_\alpha}^2 \leq \cp^2 h_{\dim\alpha}^2 \NormEnergy{v\nc}_{K_\alpha}^2.
\end{equation*}
Thus, any $L\geq \max_{T\in\Tcal} \ell(T)$ satisfies
\begin{equation*}
 \sum_{\alpha\in\{1,2\}^L} \distAlpha{\alpha} \leq h_L^2 \cp^2 \NormEnergyncK{v\nc}^2.
\end{equation*}
The combination of this result with \eqref{eqn:geometricSeriesH} -- \eqref{eqn:succEstimate} yields
\begin{align*}
 \NormLtwoK{v\nc -\intmean_K v\nc\dx}^2 &\leq  \gamma (\sum_{\ell=1}^L h_\ell^2) \NormEnergync{v\nc}{K}^2 + h_L^2 \cp^2 \NormEnergyncK{v\nc}^2\\
&\leq (\gamma (\sum_{\ell=0}^\infty h_\ell^2 - h_0^2) + h_L^2 \cp^2) \NormEnergync{v\nc}{K}^2\\
&\leq (\gamma (4M-3)h_K^2/3 + h_L^2 \cp^2) \NormEnergync{v\nc}{K}^2
\end{align*}
The passage to the limit $L\rightarrow \infty$ and $h_L\rightarrow 0$ concludes the proof with $C(n)^2=(4M-3)/(3n(n+2))$.
\end{proofof}

The remainder of this section is devoted to an alternative proof of the Poincar\'e inequality in \(2\)D in the continous case with suboptimal constant $6^{-1/2}$. The proof utilizes the techniques of the previous proof with red-refinement instead of bisection for a slightly better constant. Note that the proofs of Theorem \ref{thm:discPoin} and \ref{thm:Poin} utilize only the existence of a Poincar\'e constant $\cp$, with neither its value nor its optimality. Compared to the optimal constant $1/j_{1,1}\approx 0.26$ in \(2\)D \cite{laugesensiudeja}, the suboptimal constant $6^{-1/2}\approx 0.41$ of Theorem \ref{thm:Poin} is competitive although it utilizes elementary tools.

\begin{theorem}[Poincar\'e inequality]\label{thm:Poin}
 Let $K\subseteq \R^2$ be a triangle and $\Tcal\in \Admis(K)$ a regular triangulation of $K$. Then any $v\in H^1(K)$ satisfies
\begin{equation*}
 \NormLtwoK{v - \intmean_K v\dx} \leq \hk/ \sqrt{6} \NormEnergyOn{v}{K}.
\end{equation*}
\end{theorem}

The proof relies on the subsequent key lemma.

\begin{lemma}\label{lem:estimateDistRed}
 Any $v\in H^1(K)$ in a triangle $T\subseteq K$ and its red-refinement $\{T_1,T_2,T_3,T_4\}=\red(T)$ satisfy
\begin{equation*}
 \distSq{v}{T} \leq \max_{j=1,\dots,4}{h_{T_j}^2} \NormEnergyOn{v}{T}^2 /2 + \sum_{j=1}^4 \distSq{v}{T_j}.
\end{equation*}
\end{lemma}
\begin{proof}
\begin{figure}\label{fig:redrefinement}
\caption{Red refinement of $T$}
\centering
  \begin{tikzpicture}[scale=3]
    \draw (-1,0) node[shape=coordinate,label=below left:$P_1$] (P1) {P1};
    \draw (1,0) node[shape=coordinate,label=below right:$P_2$] (P2) {P2};
    \draw (0,1) node[shape=coordinate,label=above:$P_3$] (P3) {P3};
    \draw ($(P1)!.5!(P2)$) node[shape=coordinate,label=below:$Q_3$] (Q3) {Q3};
    \draw ($(P1)!.5!(P3)$) node[shape=coordinate,label={[label distance=-6pt]135:$Q_2$}] (Q2) {Q2};
    \draw ($(P3)!.5!(P2)$) node[shape=coordinate,label={[label distance=-6pt]40:$Q_1$}] (Q1) {Q1};
    \draw ($(Q1)!.5!(Q2)$) node[shape=coordinate,label={[label distance=-3pt]90:{\footnotesize $F_3$}}] (F3) {F3};
    \draw ($(Q1)!.5!(Q3)$) node[shape=coordinate,label={[label distance=-6pt]315:{\footnotesize $F_2$}}] (F2) {F2};
    \draw ($(Q3)!.5!(Q2)$) node[shape=coordinate,label={[label distance=-6pt]225:{\footnotesize $F_1$}}] (F1) {F1};    

    \draw ($(F1)!0.4!(P1)$) node[draw, inner sep=2pt] (T1) {$T_1$};
    \draw ($(F2)!0.4!(P2)$) node[draw, inner sep=2pt] (T2) {$T_2$};
    \draw ($(F3)!0.4!(Q3)$) node[draw, inner sep=2pt] (T3) {$T_3$};
    \draw ($(F3)!0.5!(P3)$) node[draw, inner sep=2pt] (T4) {$T_4$};

    \draw (P1) -- (P2) -- (P3) -- cycle;
    \draw (Q1) -- (Q2) -- (Q3) -- cycle;
  \end{tikzpicture}
\end{figure}

Let $F_j:= \dell T_j \cap \dell T_4$ and $Q_1,Q_2,Q_3\in\Ncal(T_4)$ with $T_4=\conv\{F_j,Q_j\}$ for $j=1,\dots,4$ as depicted in Figure \ref{fig:redrefinement}. For $j=1,2,3$, define $w_j:= \intmean_{F_j} v \ds$ and for $j=1,\dots,4$, let $\ol{v}_j:= \intmean_{T_j} v \dx$. Lemma \ref{lem:discreteTraceIdentity} -- \ref{lem:normIdMinusP} imply, for $j=1,2,3$, 
\begin{align*}
 n\abs{\ol{v}_j - w_j} \leq \frac{\sqrt{n} h_{T_j}}{\sqrt{(n+2)\abs{T_j}}}\NormEnergyOn{v}{T_j}\text{ and } n\abs{\ol{v}_4 - w_j} \leq \frac{\sqrt{n} h_{T_4}}{\sqrt{(n+2)\abs{T_4}}}\NormEnergyOn{v}{T_4}.\numberthis\label{eqn:vjminusvfRED}
\end{align*}
With $\ol{v}:=\intmean_T v \dx = (\sum_{j=1}^4 \ol{v}_j )/4$, a minimization in $\R$ and the weighted Young's inequality yield
\begin{align*}
 \sum_{j=1}^4 (\ol{v}_j-\ol{v})^2 &= \min_{x\in\R} \sum_{j=1}^4 (\ol{v}_j-x)^2 \leq \sum_{j=1}^3 (\ol{v}_j-\ol{v}_4)^2\\
 &\leq \sum_{j=1}^3 4(\ol{v}_j-w_j)^2 + 4/3 (w_j-\ol{v}_4)^2.
\end{align*}
This, the orthogonality of $v-\ol{v}_j$ onto $\ol{v}-\ol{v}_j$ in $\Ltwo{T_j}$, and $\abs{T_1}=\dots=\abs{T_4}=\abs{T}/4$ show
\begin{align*}
 \distSq{v}{T} &= \sum_{j=1}^4 \NormLtwo{v -  \ol{v}}{T_j}^2 = \sum_{j=1}^4 \NormLtwo{v -  \ol{v}_j}{T_j}^2 + \abs{T_j} \abs{\ol{v}_j- \ol{v}}^2\\
&\leq \sum_{j=1}^4 \NormLtwo{v -  \ol{v}_j}{T_j}^2 + \abs{T}/4 (\sum_{j=1}^3 (4(\ol{v}_j-w_j)^2 + 4/3 (w_j-\ol{v}_4)^2).
\end{align*}
The combination of this with \eqref{eqn:vjminusvfRED} concludes the proof.
\end{proof}

\begin{proofof}[of Theorem \ref{thm:Poin}]
 
Analogeously to the proof of Theorem \ref{thm:discPoin} but with red-refinement instead of bisection, let $\Tcal_0 := \{K\}$ and $\Tcal_\ell := \red^{(\ell)}(\Tcal_0)\in \Admis(\Tcal_0)$ for any $\ell\in\N_0$. For any multiindex $\alpha=(\alpha_1,\dots,\alpha_\ell)\in \{1,\dots,4\}^\ell$ of length $\dim\alpha=\ell\in\N_0$, define the $n$-simplex $K_\alpha$ recursively by $K_\emptyset := K$ and $\{K_{(\alpha,1)},\dots, K_{(\alpha,4)}\} = \red(K_\alpha)$ for extended multiindices $(\alpha,1),\dots,(\alpha,4) $ in $\{1,\dots,4\}^{l+1}$. This implies $\Tcal_\ell =\{K_\alpha\,|\, \dim\alpha=\ell\}$ and $h_\ell:= \max_{T\in\Tcal_\ell} h_T$ satisfies $h_{\ell+1}\leq  h_\ell/2$. Consequently,
\begin{equation}\label{eqn:geometricSeriesHred}
 \sum_{\ell=0}^{\infty} h_\ell^2 \leq h_0^2 \sum_{\ell=0}^{\infty} 4^{-\ell} =4h_0^2 /3.
\end{equation}
Successive applications of Lemma \ref{lem:estimateDistRed} as in the proof of Theorem \ref{thm:discPoin} lead to
\begin{align*}
 \NormLtwoK{v -\intmean_K v\dx} &\leq  (\sum_{\ell=1}^L h_\ell^2) \NormEnergyOn{v}{K}^2/2 + h_L^2 \cp^2 \NormEnergyOn{v}{K}^2\\
&\leq (h_K^2/6 + h_L^2 \cp^2) \NormEnergyOn{v}{K}^2
\end{align*}
The passage to the limit as $L\rightarrow \infty$ and $h_L\rightarrow 0$ concludes the proof.
\end{proofof}

The following theorem utilizes the discrete Poincar\'e inequality to prove a generalization of the error estimate for nonconforming interpolation \cite{CgallistlMorley} to nonconforming functions and also for \(n=3\). 
\begin{theorem}[Discrete Nonconforming Interpolation]
Set \(\kappa\nc^2 := C^2(n) + (n+1)\inv(n+2)\inv n^{-2}\) and let $\Inc\hat v\crr\in \CR$ with $(\Inc \hat v\crr)(\Mid(E))=\intmean_E \hat v\crr \ds$ for all $E\in\Ecal$ denote the nonconforming interpolation of the Crouzeix-Raviart function \(\hat v\crr \in CR^1_0(\Tcalf)\) on the refinement $\Tcalf\in \Admis(\Tcal)$ of \(\,\Tcal\). Then
\begin{equation*}
 h_K\inv \NormLtwoK{\hat v\crr - \Inc \hat v\crr} \leq \kappa\nc \NormEnergync{\hat v\crr- \Inc \hat v\crr}{K} \quad\text{ for any }K\in\Tcal.
\end{equation*}
\end{theorem}
\begin{proof}
Let \(M=\Mid(K)\), \(\Ecal(K)=\{E_1,\dots,E_{n+1}\}\), and \(T_j=\conv\{E_j,M\}\) for \(j=1,\dots,n+1\). Then \(\hat w\crr := (\hat v\crr - \Inc \hat v\crr)|_K \in H^1\nc(\Tcalf(K))\) satisfies \(\int_{E_j} \hat w\crr \ds = 0\) and so Lemma~\ref{lem:discreteTraceIdentity} shows
\begin{equation*}
 \hat w_K \abs{K} := \int_K \hat w\crr \dx = \sum_{j=1}^{n+1} \int_{T_j} \hat w\crr \dx = \frac1n \int_K (M-x)\cdot \nablanc \hat w\crr \dx.
\end{equation*}
%
%
 This and the discrete Poincar\'e inequality prove
\begin{align*}
 \NormLtwoK{\hat w\crr}^2 &= \NormLtwoK{\hat w\crr - \hat w_K}^2 + \abs{K} \abs{\hat w_K}^2\\
&\leq C^2(n) h_K^2 \NormEnergyncK{\hat w\crr}^2 + n^{-2} \abs{K}\inv \NormEnergyncK{\hat w\crr}^2 \NormLtwoK{\pbullet - M}^2.
\end{align*}
A modification in the proof of Lemma~\ref{lem:normIdMinusP} with \(M=0\) and therefore \(\sum_{j,k=1}^{n+1} P_j\cdot P_k=0\) proves \(\NormLtwo{\pbullet - M}{K}^2 \leq h_K^2 \abs{K}/((n+1)(n+2))\). This concludes the proof.
\end{proof}

\section{Enrichment Operator}
\label{sec:enrichment}
This section contains an interpolation estimate for a discrete interpolation operator $\Jc:\CR\rightarrow \Soz$ and the discrete Friedrichs inequality. Throughout this section, consider $n=2$.

\begin{remark}[\(3\)D case]
 The techniques of this section apply to the threedimensional case as well, but lead to more complicated constants and are not minutely detailed for brevity. The point is that there is no elementary enumeration of all simplices in a nodal patch. Therefore, the examination of different configurations leads to an eigenvalue problem with constants depending on the shape of the simplices.
\end{remark}

\begin{lemma}\label{lem:eigenvalues}
 For any $2\leq J\in\N$ and $x\in \R^J$, let $x_{J+1}:= x_1$, $\min x:= \min \{x_1,\dots, x_J\}$, and $\max x:= \max \{x_1,\dots, x_J\}$. Then it holds
 \begin{align*}
  \max_{x\in\R^J\setminus\{0\}, \min x\leq 0\leq \max x} \frac{\abs{x}^2}{\sum_{j=1}^J (x_{j+1}-x_j)^2} &=\max_{y\in\R^J\setminus\{0\}} \frac{\abs{y}^2}{\sum_{j=1}^J (y_{j+1}-y_j)^2 + (y_1+y_J)^2}\\
  &= \frac{1}{2(1-\cos(\pi/J))}.
 \end{align*}
\end{lemma}

\begin{proof}
 Define 
 \begin{align*}
  K_1 &:= \{x\in\R^J\setminus\{0\}\,|\, \min x\leq 0\leq \max x\},\\
  K_2 &:= \{x\in\R^J\setminus\{0\}\,|\, \min x= 0\},\\
  K_3 &:= \{x\in\R^J\setminus\{0\}\,|\, x_1= 0\}.
 \end{align*}
For $x\in K_1$ and $\min x\leq \mu \leq \max x$, $y:=(x_j-\mu)_{j=1,\dots,J}\in K_1$ and
\begin{equation*}
 \sum_{j=1}^J y_j^2 = \sum_{j=1}^J x_j^2 -2\mu \sum_{j=1}^J x_j + \mu^2 J.
\end{equation*}
This quadratic function of $\mu $ attains its maximum at $\min x$ or $\max x$, then
\begin{equation*}
 \frac{\abs{x}^2}{\sum_{j=1}^J (x_{j+1}-x_j)^2} \leq \frac{\max\{\abs{x-\min x}^2,\abs{x-\max x}^2\}}{\sum_{j=1}^J (x_{j+1}-x_j)^2}.
\end{equation*}
Consequently, $(x-\min x), -(x-\max x)\in K_2$ and the permutability of the indices show that
\begin{align*}
 \max_{x\in K_1} \frac{\abs{x}^2}{\sum_{j=1}^J (x_{j+1}-x_j)^2} = \max_{x\in K_2} \frac{\abs{x}^2}{\sum_{j=1}^J (x_{j+1}-x_j)^2}=\max_{x\in K_3} \frac{\abs{x}^2}{\sum_{j=1}^J (x_{j+1}-x_j)^2}.
\end{align*}

Furthermore, any $x\in K_3$ satisfies $\sum_{j=1}^J (x_{j+1}-x_j)^2=x_2^2 + \sum_{j=2}^{J-1} (x_{j+1}-x_{j})^2 + x_{J}^2=\tilde x\cdot A\tilde x$ with $\tilde x=(x_2,\dots,x_J)$ and the tridiagonal $(J-1)\times(J-1)$ matrix
\begin{equation*}
 A=\begin{pmatrix}
    2&-1&&\\
    -1&\ddots&\ddots&\\
    &\ddots&\ddots&-1\\
    &&-1&2
   \end{pmatrix}\in\R^{(J-1)\times (J-1)}.
\end{equation*}
A direct calculation with the trigonometric addition formulas for the sine function shows that for any $k=1,\dots,J-1$, the vector $x^k$ with components $x^k_j=\sin(kj\pi/J)$ is an eigenvector of $A$ with eigenvalue $\lambda_k:=2(1-\cos(k\pi/J))>0$ \cite[Thm.~3.2(v)]{MR2440234}. Since $0<\lambda_1<\dots<\lambda_{J-1}$, $A$ is positive definite and $\lambda_1\abs{x}^2 =\lambda_1\abs{\tilde x}^2\leq \tilde x\cdot A\tilde x$ concludes the proof of the first equality.
 
For the second equality, observe that any $y\in \R^J\setminus\{0\}$ satifies $\sum_{j=1}^J (y_{j+1}-y_j)^2 + (y_1+y_J)^2= y\cdot By$ with the tridiagonal matrix
\begin{equation*}
 B=\begin{pmatrix}
    3&-1&&&\\
    -1&2&\ddots&&\\
    &\ddots&\ddots&\ddots&\\
    &&\ddots&2&-1\\
    &&&-1&3
   \end{pmatrix}\in\R^{J\times J}.
\end{equation*}
A straight-forward calculation shows that the vectors $y^k\in\R^J$ with components $y^k_j=(1+\cos(k\pi/J))\sin(kj\pi/J)- \sin(k\pi/J) \cos(kj\pi/J)$ for $k=1,\dots,J-1$ and $y^J_j=\cos(j\pi/J)$ are eigenvectors of $B$ with eigenvalues $\lambda_k:=2(1-\cos(k\pi/J))>0$ for $k=1,\dots,J$ \cite[Thm.~3.4(iii)]{MR2440234}. Consequently, $\lambda_1\abs{y}^2 \leq y\cdot By$.
\end{proof}

Let $\vcr\in\CR$ and $\vc:=\Jc(\vcr)\in \Soz$ with
\begin{equation}\label{eqn:convCondition}
 \vc(z) \in \conv\{ \vcr|_T(z) \,|\, T\in\Tcal(z) \}\quad\text{ for any }z\in\Ncal(\Omega).
\end{equation}
The shape regularity of $\Tcal$ leads to a minimum angle $\omz$ in $\Tcal$, i.e. $0<\omz\leq \min\sphericalangle\Tcal$. Let $\cinner:=\max\{\abs{\Tcal(z)}\,|\, z\in\Ncal(\Om)\}\geq 2$, $\cbd:=\max\{\abs{\Tcal(z)}\,|\, z\in\Ncal(\bOm)\}$, $\cpatch:= \max\{\cinner,\cbd\}$, and define $\cjc^2=(\sqrt{3}/2)\cot(\omz)/(1-\cos(\pi/\cpatch))$.
\begin{remark}
 The estimate $(1-\cos(x))\inv \leq 4/x^2$ for $0<x\leq \pi/2$ leads to the simpler estimate
\begin{equation*}
 \cjc\leq (2\sqrt{3}\cot(\omz))\poh \cpatch/\pi.
\end{equation*}
\end{remark}
\begin{remark}
For the case of a triangulation with right isosceles triangles, $\cjc = (\sqrt{3}/(2-2\cos(\pi/8)))\poh \leq 3.3729$.
\end{remark}

\begin{theorem}[Interpolation error for $\Jc$]\label{thm:interror}
 Any interpolation operator $\Jc:\CR\rightarrow \Soz$ with \eqref{eqn:convCondition} satisfies
\begin{equation*}
 \NormLtwoOm{\htc\inv(1-\Jc)\vcr} \leq \cjc \NormEnergyncOnly{\vcr}.
\end{equation*}
This estimate also holds for any $T\in\Tcal$ in that
\begin{equation*}
 \NormLtwo{h_T\inv(1-\Jc)\vcr}{T} \leq \cjc \NormEnergync{\vcr}{\omega_T}.
\end{equation*}
\end{theorem}
\begin{proof}
 For any $T\in\Tcal$ and $z\in\Ncal(T)$, let $e_T(z):= \vcr|_T(z) - \vc(z)$ and $e(z)^2:= \sum_{T\in\Tcal(z)} e_T(z)^2$. With $e_T:=(e_T(z))_{z\in\Ncal(T)}\in\R^3$, a direct calculation with mass matrix 
 \begin{equation*}
  M=\frac{\abs{T}}{12}\begin{pmatrix}
     2&1&1\\ 1&2&1\\ 1&1&2
    \end{pmatrix}
\in\R^{3\times 3}
 \end{equation*}
of the barycentric coordinates with eigenvalues $\abs{T}/12$ and $\abs{T}/3$ and the estimate $\abs{T}\leq \sqrt{3}h_T^2/4$ shows
\begin{equation}\label{eqn:LHSvse}
 h_T^{-2} \NormLtwoT{\vcr-\vc}^2 =h_T^{-2} e_T\cdot M e_T \leq \abs{T}/(3h_T^2) \abs{e_T}^2 \leq 1/(4\sqrt{3}) \sum_{z\in\Ncal(T)} e_T(z)^2.
\end{equation}
Any $T\in\Tcal$ and $p_1\in P_1(T)$ satisfy 
\begin{equation*}
 \max_{z_1,z_2\in\Ncal(T)} \abs{p_1(z_1)-p_1(z_2)}^2 \leq h_T^2/\abs{T}\NormEnergyOn{p_1}{T}^2.
\end{equation*}
This, $h_T^2/\abs{T}\leq 4\cot(\omz)$ and the triangle inequality show that any $\partial T_+\cap \partial T_- \in\Ecal(\Om)$ with $z\in\Ncal(E)$ and $T_\pm\in\Tcal$ satisfies
\begin{align*}
 \abs{e_{T_+}(z) - e_{T_-}(z) } &= \abs{\vcr|_{T_+}(z) - \vcr|_{T_-}(z) }\\
&\leq \abs{\vcr|_{T_+}(z) - \vcr(\Mid(E)) } + \abs{\vcr(\Mid(E)) - \vcr|_{T_-}(z) }\\
&\leq 1/2 \max_{z_1,z_2\in\Ncal(T_+)} \abs{\vcr|_{T_+}(z_1)-\vcr|_{T_+}(z_2)}\\
&\quad\quad\quad +  1/2 \max_{z_1,z_2\in\Ncal(T_-)} \abs{\vcr|_{T_-}(z_1)-\vcr|_{T_-}(z_2)} \\
&\leq \cot(\omz)\poh (\NormEnergyOn{\vcr}{T_+} +\NormEnergyOn{\vcr}{T_-} )\\
&\leq (2\cot(\omz))\poh \NormEnergync{\vcr}{\omega_E}.\numberthis\label{eqn:errordifference}
\end{align*}
Analogeously, $E\in\Ecal(\bOm)$ with $T\in\Tcal$, $E\in\Ecal(T)$, and $z\in\Ncal(E)$ satisfies $\abs{e_T(z)}\leq \cot(\omz)\poh \NormEnergync{\vcr}{T}$.

Consider $z\in\Ncal(\bOm)$ with $\Tcal(z)=\{T_1,\dots,T_J\}$ and $E_1:=\partial T_1\cap \bOm$, $E_{J+1}:=\partial T_J\cap \bOm$, and $E_{j+1}:=\partial T_j\cap \partial T_{j+1}\in\Ecal(\Om)$ for $j=1,\dots,J-1$. With $e_j:=e_{T_j}(z)$ for $j=1,\dots,J$ and $e_{J+1}:=e_1$, the previous estimates show that $\abs{e_j}^2\leq \cot(\omz) \NormEnergync{\vcr}{T_j}^2$ for $j=1,J$ and $\abs{e_j-e_{j+1}}^2\leq 2\cot(\omz) \NormEnergync{\vcr}{\omega_{E_{j+1}}}^2$ for $j=1,\dots,J-1$. Hence
\begin{equation}\label{eqn:eMeUpper}
  \abs{e_1 + e_J}^2 + \sum_{j=1}^{J} \abs{e_{j+1}-e_j}^2  = 2\abs{e_1}^2 + \sum_{j=1}^{J-1} \abs{e_{j+1}-e_j}^2 + 2\abs{e_J}^2\leq 4 \cot(\omz) \NormEnergync{\vcr}{\omega_z}^2.
\end{equation}
This and Lemma \ref{lem:eigenvalues} show that $e=(e_1,\dots,e_J)^\top\in \R^J$ satisfies 
\begin{align*}
 e(z)^2 = \abs{e}^2 \leq 2 \cot(\omz)/(1-\cos(\pi/J)) \NormEnergync{\vcr}{\omega_z}^2.
\end{align*}

For $z\in\Ncal(\Om)$ with $\Tcal(z)=\{T_1,\dots,T_J\}$, $T_{J+1}:=T_{1}$ and $ \partial T_j\cap \partial T_{j+1}\in\Ecal(\Om)$ for $j=1,\dots,J-1$ and $\partial T_J\cap \partial T_{1}\in\Ecal(\Om)$, \eqref{eqn:errordifference} shows that $\abs{e_j-e_{j+1}}^2\leq 2\cot(\omz) \NormEnergync{\vcr}{T_j\cup T_{j+1}}^2$ for $j=1,\dots,J$. Since $0\in\conv\{e_1,\dots,e_J\}$, it follows $\min e \leq 0 \leq \max e$ and Lemma \ref{lem:eigenvalues} leads to
\begin{align}\label{eqn:ez4inner}
 e(z)^2 = \abs{e}^2 \leq 2 \cot(\omz)/(1-\cos(\pi/J)) \NormEnergync{\vcr}{\omega_z}^2.
\end{align}
Altogether, any $z\in\Ncal$ satisfies
\begin{align*}
 e(z)^2 \leq 2 \cot(\omz)/(1-\cos(\pi/\cpatch)) \NormEnergync{\vcr}{\omega_z}^2=: 4/\sqrt{3}\cjc^2 \NormEnergync{\vcr}{\omega_z}^2.
\end{align*}
This, \eqref{eqn:LHSvse}, and an overlapping argument show the local estimate
\begin{equation*}
 h_T^{-2} \NormLtwoT{\vcr-\vc}^2 \leq 1/(4\sqrt{3}) \sum_{z\in\Ncal(T)} e(z)^2\leq \sum_{z\in\Ncal(T)}\cjc^2/3  \NormEnergync{\vcr}{\omega_z}^2 \leq \cjc^2 \NormEnergync{\vcr}{\omega_T}^2.
\end{equation*}
The sum over all $T\in\Tcal$ and the previous arguments lead to
\begin{align}\label{eqn:finalEst}
 \NormLtwoOm{\htc\inv(\vcr-\vc)}^2 \leq \cjc^2/3 \sum_{z\in\Ncal} \NormEnergync{\vcr}{\omega_z}^2= \cjc^2 \NormEnergyncOnly{\vcr}^2.
\end{align}
\end{proof}

\begin{examples}\label{rem:examples}
\begin{enumerate}
 \item One example of $\Jc:\CR\rightarrow \Soz$ with \eqref{eqn:convCondition} is the enrichment operator $\Jc :=J_1$ \cite[p.~297]{brennerscott} with
\begin{equation}\label{eqn:defj1}
 J_1\vcr(z) := \abs{\Tcal(z)}\inv \sum_{T\in\Tcal(z)} (\vcr|_T)(z) \quad\text{ for any }z\in\Ncal(\Omega).
\end{equation}
\item Another is the (possibly new) precise representation $\Jc\vcr := I_C \vcr^\star$ with
\begin{equation}
  I_C\vcr^\star(z) := (2\pi)\inv \sum_{T\in\Tcal(z)} \measuredangle(T,z) (\vcr|_T)(z) \quad\text{ for any }z\in\Ncal(\Omega).
\end{equation}
\item Other examples are the maximum or minimum at each node,
\begin{align*}
  \Jc\vcr(z) &:= \max_{T\in\Tcal(z)} (\vcr|_T)(z) \quad\text{ for any }z\in\Ncal(\Omega)\text{ or }\\
  \Jc\vcr(z) &:= \min_{T\in\Tcal(z)} (\vcr|_T)(z) \quad\text{ for any }z\in\Ncal(\Omega).
\end{align*}
\item A discrete quasi-interpolation for the proof of optimal convergence rates of adaptive methods motivates the next example in a general formulation here. In the context of adaptive methods, $\mathcal{U}=\Tcal\cap\Tcalf\ins\Tcal$ for a triangulation $\Tcal$ and refinement $\Tcalf$, see Remark \ref{rem:dQI}. 
In a general setting, let $\vcr\in\CR$ and suppose there exists $\mathcal{U}\ins\Tcal$ such that for any $K_1,K_2\in\mathcal{U}$ with a shared node $z\in\Ncal(K_1)\cap\Ncal(K_2)$, the value of $\vcr$ at $z$ coincide, e.g. $\vcr|_{K_1}(z) = \vcr|_{K_2}(z)$. Hence, $J_{QI}\vcr\in\Soz$ is well-defined and satisfies \eqref{eqn:convCondition} for
\begin{equation}\label{eqn:defjQI}
 J_{QI}\vcr(z) := \begin{cases}
                   \vcr|_{K}(z)&\text{ if there exists }K\in\mathcal{U}\text{ with }z\in\Ncal(K),\\
		   J_1\vcr(z)&\text{ else.}
                  \end{cases}
\end{equation}
\end{enumerate}

\end{examples}

\begin{remark}
 Similar calculations with $2\abs{e_{T_+}(z) - e_{T_-}(z) }\leq \eta_E :=\abs{E}\abs{\jump{\partial\vcr/\partial s}_E}$ for $E\in\Ecal(\Omega)$ in \eqref{eqn:errordifference}, $2\abs{e_{T}(z)}\leq \eta_E$ for $E\in\Ecal(\partial\Omega)$,  and $\sum_{E\in\Ecal} \eta_E^2 \leq 30\cot(\omz) \NormEnergyncOnly{\vcr-v}$ for any $v\in H^1_0(\Omega)$ lead to a generalized version of Theorem \ref{thm:interror} with $C_1^2=15\cot(\omz)/(8\sqrt{3} \min\{1-\cos (\pi/\cinner), 1-\cos(\pi/(\cbd+1))\})$,
\begin{equation*}
 \NormLtwoOm{\htc\inv(1-\Jc)\vcr} \leq C_1 \min_{v\in H^1_0(\Omega)} \NormEnergyncOnly{\vcr-v}.
\end{equation*}

\end{remark}

\begin{lemma}\label{lem:caseJ1}
 For the special case $\Jc=J_1$ from \eqref{eqn:defj1}, an improved constant in the estimate of Theorem \ref{thm:interror} reads
\begin{equation*}
 \cjc(J_1)^2=(\sqrt{3}/2)\cot(\omz)/ \min\{1-\cos (2\pi/\cinner), 1-\cos(\pi/\cbd)\}.
\end{equation*}
\end{lemma}
\begin{proof}
 The only change with respect to the proof of Theorem \ref{thm:interror} concerns the estimate \eqref{eqn:ez4inner} of $e(z)^2$ for inner nodes $z\in\Ncal(\Om)$. Recall that for  $z\in\Ncal(\Om)$ with patch $\Tcal(z)=\{T_1,\dots,T_J\}$ and $e_j=\vcr|_{T_j}(z)-\vc(z)$ for $j=1,\dots,J$, \eqref{eqn:errordifference} shows
\begin{equation*}
 \abs{e_j-e_{j+1}}^2\leq 2\cot(\omz) \NormEnergync{\vcr}{T_j\cup T_{j+1}}^2\text{ for }j=1,\dots,J
\end{equation*}
(with $e_{J+1}:= e_0$ and $T_{J+1}:= T_0$). Define $e=(e_1,\dots,e_J)^\top\in\R^J$ and 
\begin{equation*}
 C=\begin{pmatrix}
    2&-1&&&-1\\
    -1&2&\ddots&&\\
    &\ddots&\ddots&\ddots&\\
    &&\ddots&2&-1\\
    -1&&&-1&2
   \end{pmatrix}\in\R^{J\times J}.
\end{equation*}
Consequently, 
\begin{equation}\label{eqn:eCe}
 e\cdot Ce = \sum_{j=1}^{J} \abs{e_j-e_{j+1}}^2 \leq 4\cot(\omz) \NormEnergync{\vcr}{\omega_z}^2.
\end{equation}

For an approach similar to the one in the proof of Lemma \ref{lem:eigenvalues}, compute the eigenvalues $0=\lambda_0<\lambda_1<\dots<\lambda_K$ of the matrix $C\in \R^{J\times J}$ with $K:=\lfloor J/2 \rfloor$ with floor function $\lfloor \pbullet\rfloor$ (i.e. $K=J/2$ for even $J$ and $K = (J-1)/2$ for odd $J$), $\lambda_k=2-2\cos( 2k\pi/J)$ \cite[Thm.~3.4(v)]{MR2440234} for $k=0,\dots,\lfloor J/2\rfloor$. Indeed, the trigonometric addition formulae for sine and cosine show that the vectors $x^K,y^K\in\R^J$ with $x^k_j=\cos(2jk\pi/J)$, $y^k_j=\sin(2jk\pi/J)$ for $j=1,\dots,J$, are the $0$-vector or non-zero eigenvectors of $C$ with eigenvalue $\lambda_k$ for $k=0,\dots,K$. An analysis of linear independence of $x^k,y^k\neq 0$ for even and odd $J$ shows that there are $J$ linearly independent eigenvectors. In any case, $C$ is positive semi-definite with eigenvalues $0=\lambda_0<\lambda_1<\dots<\lambda_K$ and $\lambda_0=0$ is a simple eigenvalue with the eigenvector $u=(1,\dots,1)^\top$ that is orthogonal to all other eigenvectors of $C$.

The identities $e=(\vcr|_{T_1}(z),\dots,\vcr|_{T_J}(z))^\top-\vc(z) u$ and the definition of $\vc(z)$ imply the orthogonality $e\cdot u = 0$. Hence, $\lambda_1\abs{e}^2\leq e\cdot Ce$ and therefore \eqref{eqn:eCe} shows
\begin{align*}
 e(z)^2 = \abs{e}^2 \leq (4\cot(\omz)/\lambda_1) \NormEnergync{\vcr}{\omega_z} = (2\cot(\omz)/(1-\cos 2\pi/J)) \NormEnergync{\vcr}{\omega_z}.
\end{align*}
The remaining parts of the proof of Theorem \ref{thm:interror} apply verbatim with different constants.
\end{proof}
%
\begin{example}
For the case of a triangulation of a convex domain with right isosceles triangles, $\cjc(J_1)= (\sqrt{3}/(2-2\cos(\pi/4)))\poh \leq 1.6002$.
\end{example}

The use of this discrete interpolation estimate enables a proof of the discrete Friedrichs inequality and an interpolation estimate for a new quasi-interpolation operator $J:H^1_0(\Om)\rightarrow S^1_0(\Tcal)$ with the help of an inverse estimate.

\begin{lemma}[inverse estimate]\label{lem:invEst}
 Any $T\in\Tcal$, $p_1\in P_1(T)$, and the constant 
\begin{equation*}
 \cinv^2 := 24 \cot(\omz) (2\cot(\omz)-\cot(2\omz) + ((2\cot(\omz)-\cot(2\omz))^2-3)\poh)
\end{equation*}
satisfy
\begin{equation*}
 \NormEnergyOn{p_1}{T} \leq \cinv h_T\inv\NormLtwo{p_1}{T}.
\end{equation*}
\end{lemma}
\begin{proof}
 An analysis of the eigenvalues of the stiffness and the mass matrix and $\sigma=\sum_{z\in\Ncal(T)} \cot(\measuredangle(T,z))$ leads to the local inverse estimate
\begin{equation*}
  \NormEnergyOn{p_1}{T}^2 \leq 6(\sigma + \sqrt{\sigma^2 -3})/\abs{T} \NormLtwoT{p_1}^2.
\end{equation*}
A maximization shows $\sigma\leq 2\cot(\omz)-\cot(2\omz)$ and $1/\abs{T}\leq h_T^{-2} 4\cot(\omz)$ concludes the proof.
\end{proof}
For right isosceles triangles, the constant $\cinv=\sqrt{72}$ and all estimates in the proof are sharp.

\begin{coro}[discrete Friedrichs inequality]\label{cor:dF}
Any $\vcr\in\CR$ and the constants $\cdf = \hmax\cjc(J_1) + \cf(1+\cinv \cjc(J_1))$ and $\cf=\operatorname{width}(\Omega)/\pi$  satisfy
 \begin{equation*}
  \NormLtwoOm{\vcr} \leq \cdf \NormEnergyncOnly{\vcr}.
 \end{equation*}
\end{coro}
\begin{proof}
 Given $\vcr\in\CR$, let $\vc=J_1(\vcr)$ for the enrichment operator $J_1$ from Remark \ref{rem:examples} so that Lemma \ref{lem:caseJ1} shows 
\begin{equation*}
 \NormLtwoOm{\vcr -\vc} \leq \hmax\cjc(J_1) \NormEnergyncOnly{\vcr}.
\end{equation*}
Lemma \ref{lem:invEst}, the Friedrichs inequality $\NormLtwoOm{\vc}\leq \diam(\Om) \NormEnergy{\vc}/\pi$, and the triangle inequality yield
\begin{align*}
 \NormLtwoOm{\vc}&\leq \cf\NormEnergyncOnly{\vc} \leq\cf (\NormEnergy{\vcr} + \cinv \NormLtwoOm{\htc\inv (\vc-\vcr)})\\
&\leq \cf(1+\cinv \cjc(J_1))\NormEnergyncOnly{\vcr}.\numberthis\label{eqn:stabEst}
\end{align*}
The triangle inequality $\NormLtwoOm{\vcr} \leq \NormLtwoOm{\vcr -\vc} + \NormLtwoOm{\vc}$ concludes the proof.
\end{proof}

\section{Quasi-Interpolation}\label{sec:quasiinterpolation}
This section proves an estimate for a quasi-interpolation operator $J:H^1_0(\Om)\rightarrow S^1_0(\Tcal)$ as conclusion of the enrichment operator of Section~\ref{sec:enrichment}.
For $n=2$, let $\Inc:H^1_0(\Om)\rightarrow \CR$ denote the non-conforming interpolation operator with $(\Inc v)(\Mid(E))=\intmean_E v \ds$ for all $E\in\Ecal$ and $v\in H^1_0(\Om)$.

\begin{theorem}[Quasi-interpolation]\label{thm:QI}
 The bounded linear projection $J:= J_C\circ \Inc: H^1_0(\Om)\rightarrow S^1_0(\Tcal)$ for any $\Jc:\CR\rightarrow \Soz$ with \eqref{eqn:convCondition} and any $v\in H^1_0(\Omega)$ satisfy
\begin{align*}
 \NormLtwoOm{\htc\inv(1-J)v} &\leq (\kappa^2+\cjc^2)^{1/2} \NormEnergy{v}\text{ and }\\
 \NormEnergy{Jv},\NormEnergy{(1-J)v} &\leq \cf(1+\cinv \cjc)\NormEnergy{v}
\end{align*}
with the constant $\kappa=(1/48+1/j_{1,1}^2)^{1/2}$ and the first positive root $\joo$ of the Bessel function of the first kind. 
Additionally, for any $T\in\Tcal$, $f|_{\omega_T}\in S^1(\Tcal(\omega_T))$ implies 
\begin{equation}\label{eqn:localProj}
 f|_T = (Jf)|_T.
\end{equation}
With $\csecond := (\kappa+1)/\joo + (1+\cinv)\com \cjc (1/\joo + \calph)$, $\com = \sin(\omz)^{-\max\{\cbd-1,\cinner/2\}}$, $\calph=\max_{T\in\Tcal,z\in\Ncal(T)}((1/4+2/j_{1,1}^2)/(1-\abs{\cos(\measuredangle(T,z))}))\poh$, any $v\in H^2(\Om)\cap H^1_0(\Om)$ additionally satisfies the second-order approximation property
\begin{equation*}
 \NormLtwoOm{\htc^{-2}(1-J)v} + \NormLtwoOm{\htc\inv \nabla((1-J)v)} \leq \csecond \NormLtwoOm{\D^2 v}.
\end{equation*}

\end{theorem}
\begin{proof}
 For the proof of the first estimate, the triangle inequality implies 
\begin{equation*}
 \NormLtwoOm{\htc\inv(v - \Jc \Inc v)} \leq \NormLtwoOm{\htc\inv(v - \Inc v)} + \NormLtwoOm{\htc\inv(1 - \Jc) \Inc v)}.
\end{equation*}
The interpolation estimate for the non-conforming interpolation operator with $\kappa=(1/48+1/j_{1,1}^2)^{1/2}= 0.29823$ \cite{CgallistlMorley}, Theorem~\ref{thm:interror},  and the orthogonality of $\nablanc(v-\Inc v)$ onto $\nablanc \Inc v$ in $L^2(\Om)$ yield
\begin{equation*}
 \NormLtwoOm{\htc\inv(v - \Jc \Inc v)} \leq (\kappa^2+\cjc^2)^{1/2} \NormEnergy{v}.
\end{equation*}

For the second estimate, observe that $J:H^1_0(\Om) \rightarrow H^1_0(\Om)$ is a projection in $\big(H^1_0(\Om), (\nabla\pbullet,\nabla\pbullet)_{L^2(\Om)}\big)$ and thus, $\norm{1-J}_{L(H^1_0(\Om);H^1_0(\Om))}=\norm{J}_{L(H^1_0(\Om);H^1_0(\Om))}$ \cite{kato}. Consequently, \eqref{eqn:stabEst} from the proof of the discrete Friedrichs inequality and\\ $\NormEnergyncOnly{\Inc v} \leq \NormEnergy{v}$ show
\begin{align*}
 \norm{J}_{L(H^1_0(\Om);H^1_0(\Om))} \leq  \cf(1+\cinv \cjc).
\end{align*}

For $T\in\Tcal$ and $f|_{\omega_T}\in S^1(\Tcal(\omega_T))$ as in \eqref{eqn:localProj}, any $z\in\Ncal(T)$ satisfies
\begin{equation*}
 (\Jc(\Inc f)) (z) = (\Jc (\Inc f|_{\omega_z})) (z)= (\Jc(f|_{\omega_z})) (z) = f(z).
\end{equation*}

For the proof of the second-order approximation property, let $v\in H^2(\Om)\cap H^1_0(\Om)$ and $Iv\in S^1_0(\Tcal), Iv(z)=v(z)$ the nodal interpolant. $(1-\Jc)Iv = 0$ implies  $(1-J)v= (1-\Inc)v + (1-\Jc)(\Inc v - I v)$. The triangle inequality yields
\begin{align*}
 \NormLtwoOm{\htc^{-2}(1-J)v}&\leq \NormLtwoOm{\htc^{-2}(1-\Inc)v} + \NormLtwoOm{\htc^{-2}(1-\Jc)(\Inc v - I v)}.
\end{align*}
The second-order interpolation errors of non-conforming \cite{CgallistlMorley} and nodal interpolation \cite{CGedickeRim} read
\begin{align*}
 \NormLtwoOm{\htc^{-2}(1-\Inc)v} &\leq \kappa \NormLtwoOm{\htc^{-1}\nablanc(1-\Inc)v} \leq \kappa/\joo \NormLtwoOm{\D^2 v},\\
 \NormLtwoOm{\htc^{-1}\nabla(1-I)v} &\leq \calph \NormLtwoOm{\D^2 v}.
\end{align*}
Consequently, a slight modification of the proof of Theorem \ref{thm:interror} in \eqref{eqn:finalEst} with the estimate $h_T \leq \max_{K\in \Tcal(z)} h_K \leq \com h_T$ for any $z\in\Ncal, T\in \Tcal(z)$, and a triangle inequality implies
\begin{align*}
 \NormLtwoOm{\htc^{-2}(1-\Jc)(\Inc v - I v)} &\leq \com \cjc \NormLtwoOm{\htc^{-1}\nablanc(\Inc v - I v)}\\
& \leq  \com \cjc (1/\joo + \calph) \NormLtwoOm{\D^2 v}.
\end{align*}
This results in the estimate of the first term in the assertion,
\begin{equation*}
 \NormLtwoOm{\htc^{-2}(1-J)v} \leq (\kappa/\joo + \com \cjc (1/\joo + \calph)) \NormLtwoOm{\D^2 v}.
\end{equation*}
The split from above yields 
\begin{equation*}
 \NormLtwoOm{\htc\inv \nabla((1-J)v)} \leq  \NormLtwoOm{\htc\inv \nabla((1-\Inc)v)} +  \NormLtwoOm{\htc\inv \nabla((1-\Jc)(\Inc v - I v))}.
\end{equation*}
The inverse estimate leads to $\NormLtwoOm{\htc\inv \nabla((1-\Jc)(\Inc v - I v))} \leq \cinv \NormLtwoOm{\htc^{-2}(1-\Jc) \allowbreak (\Inc v - I v)}$ and therefore
\begin{equation*}
 \NormLtwoOm{\htc\inv \nabla((1-J)v)} \leq  (1/\joo +  \cinv \com \cjc (1/\joo + \calph))\NormLtwoOm{\D^2 v}.
\end{equation*}

\end{proof}

\begin{remark}[Discrete quasi-interpolation]\label{rem:dQI}
Consider a triangulation $\Tcal$ and refinement $\Tcalf$. For any $\hat v_C\in\Sozf$ and $K\in \mathcal{U}:= \Tcal\cap\Tcalf$, $\Inc\hat v_C|_K = \hat v_C|_K$. Hence, any $K_1,K_2\in \mathcal{U}$ with $z\in\Ncal(K_1)\cap\Ncal(K_2)$ satisfy $\Inc\hat v_C|_{K_1}(z) = \hat v_C (z) = \Inc\hat v_C|_{K_2}(z)$. Consequently, the application of Theorem \ref{thm:QI} with $\Jc = J_{QI}$ from \eqref{eqn:defjQI} yields a discrete quasi-interpolation $J_{dQI}:= J_{QI}\circ \Inc|_{\Sozf}:\Sozf\rightarrow\Soz$ such that any $\hat v_C\in\Sozf$ satisfies $\hat v_C= J_{dQI} \hat v_C$ on $\Tcal\cap\Tcalf$ and
\begin{align}\label{eqn:dQIApprox}
 \NormLtwoOm{\htc\inv(1-J_{dQI})\hat v_C} &\leq (\kappa^2+\cjc^2)^{1/2} \NormEnergy{\hat v_C}.
\end{align}
 A thorough inspection of the proofs of Theorems \ref{thm:interror} and \ref{thm:QI} shows that this interpolation operator can be extended to $J_{dQI}:\Sof\rightarrow\So$ with the same properties and constant \(\cjc^2=(\sqrt{3}/2)\cot(\omz)/ \min\{1-\cos (\pi/\cinner), 1-\cos(\pi/(2\cbd-1))\}\) arising from the eigenvalue problem \cite[Thm.~3.2(viii)]{MR2440234}.
\end{remark}

\section{Constants in the Axioms of Adaptivity}\label{sec:axioms}
This section recapitulates the proof of optimal convergence rates of the Courant and the Crouzeix-Raviart FEM in \(2\)D in the axiomatic framework of \cite{axioms, safem} with explicit constants. Define $a(u,v):= \allowbreak (\nabla u,\nabla v)_{L^2(\Om)}$ for any $v,w\in H^1_0(\Omega)$. Given $f\in\Ltwo{\Om}$, the CFEM seeks $u_C\in\Soz$ with
\begin{equation}\label{eqn:C}
 a(u_C,v_C)=(f,v_C)_{L^2(\Om)}\quad\text{ for any }v_C\in\Soz.
\end{equation}
For any admissible triangulation $\Tcal\in\Admis$ with CFEM solution $u_C\in\CR$ to \eqref{eqn:C} and $K\in\Tcal$, define
\begin{equation*}
 \eta_C^2(\Tcal,K):= \abs{K} \NormLtwoK{f}^2 + \abs{K}\poh \sum_{E\in\Ecal(K)\cap \Ecal(\Omega)} \NormLtwo{\jump{\nabla u_C\cdot\nu_E}}{E}^2.
\end{equation*}
For $\Tcal\in\Admis$ and refinement $\Tcalf$ with solutions $u_C\in\Soz$ and $\hat u_C\in\Sozf$, define
\begin{equation*}
 \delta_C(\Tcal,\Tcalf):= \NormEnergy{u_C-\hat u_C}.
\end{equation*}
The optimality proof of \cite{axioms} relies on the axioms (A1)--(A4) below with constants $0<\Lambda_1, \Lambda_2, \Lambda_3, \Lambda_4<\infty$ and $0<\varrho_2<1$.
Any $\Tcal\in\Admis$ and refinement $\Tcalf$ satisfy Stability (A1)
\begin{equation}\label{eqn:stab}
 \abs{\eta_C(\Tcal,\Tcal\cap \Tcalf)-\eta_C(\Tcalf,\Tcal\cap\Tcalf)} \leq \Lambda_1 \delta_C(\Tcal,\Tcalf)
\end{equation}
and Reduction (A2)
\begin{equation*}
 \eta_C(\Tcalf,\Tcalf\setminus\Tcal) \leq \varrho_2 \eta_C(\Tcal,\Tcal\setminus\Tcalf) + \Lambda_2 \delta_C(\Tcal,\Tcalf).
\end{equation*}


Moreover, \cite{axioms} shows discrete reliability (A3) on a simply-connected domain $\Om\ins\R^2$,
\begin{equation}\label{eqn:dRel}
  \delta_C^2(\Tcal,\Tcalf) \leq \Lambda_3 \eta_C^2(\Tcal,\Tcal\setminus\Tcalf).
 \end{equation}
The quasi-orthogonality (A4) shows that the output $\Tcal_k$, $k=1,2,\dots$ of the adaptive algorithm with corresponding quantities $\eta_k:=\eta_C(\Tcal_k,\Tcal_k)$ and any $\ell,m\in \N$ satisfy
 \begin{equation*}
  \sum_{k=\ell}^{\ell+m}\delta_C^2(\Tcal_k,\Tcal_{k+1}) \leq  \Lambda_4 \eta_\ell^2.
 \end{equation*}
The main result \cite[Theorem 4.5]{axioms} and the axioms of adaptivity state that (A1)--(A4) with the above-mentioned constants yield optimal convergence rates of the adaptive Crouzeix-Raviart FEM with Dörfler marking for any bulk parameter
\begin{equation}\label{eqn:theta}
 0<\theta<\theta_0:= (1+\Lambda_1^2\Lambda_3)\inv.
\end{equation}
This is a sufficient condition for optimal rates and requires the quantification of \(\theta_0\) and so to calculate $\Lambda_1$ and $\Lambda_3$ explicitly.

The proof of stability (A1) is essentially contained in \cite{CKNS08} but is included here for explicit gathering of the constants.
\begin{theorem}[Stability (A1) for CFEM]\label{thm:stabC}
 The constants\\ $\cquot:=\max_{K_1,K_2\in\Tcal, \Ecal(K_1)\cap\Ecal(K_2)\neq\emptyset} \abs{K_1}/\abs{K_2}\leq 2\cot(\omz)/\sin(\omz)$ and $\Lambda_1^2= 6\cot\poh(\omz) \allowbreak(1+\cquot\poh)^2$ satisfy \eqref{eqn:stab}.
\end{theorem}
\begin{proof}
 The reverse triangle inequality for vectors with entries $\abs{T}^{1/4} \NormLtwo{\jump{\nabla u_C\cdot\nu_E}_E}{E}$ resp. $\abs{T}^{1/4} \NormLtwo{\jump{\partial\hat u_C/\partial s}_E}{E}$ for any $T\in\Tcal\cap\Tcalf$ and $E\in\Ecal(T)$ shows
 \begin{align*}
  \abs{\eta_C(\Tcal,\Tcal\cap \Tcalf)&-\eta_C(\Tcalf,\Tcal\cap\Tcalf)}^2\\
  &\leq \sum_{T\in \Tcal\cap\Tcalf} \sum_{E\in\Ecal(T)} \abs{T}\poh \big(\NormLtwo{\jump{\nabla u_C\cdot\nu_E}_E}{E} - \NormLtwo{\jump{\nabla \hat u_C\cdot\nu_E}_E}{E}\big)^2.
 \end{align*}
Furthermore, the reverse triangle inequality in $L^2(E)$ imply that any $T\in\Tcal\cap\Tcalf$ and $E\in\Ecal(T)$ satisfy
\begin{align*}
 \big\vert \NormLtwo{\jump{\nabla u_C\cdot\nu_E}_E}{E} - \NormLtwo{\jump{\nabla \hat u_C\cdot\nu_E}_E}{E} \big\vert&\leq \NormLtwo{\jump{\nablanc (u_C-\hat u_C)}_E}{E}.
\end{align*}
The triangle inequality and the trace identity shows that $\hat p_0:=\nablanc (u_C-\hat u_C)\in P_0(\Tcalf;\R^2)$ satifies on $\partial T_+ \cap\partial T_-=E\in\Ecalf(\Om)$ with $\hat T_+,\hat T_-\in\Tcalf$,
\begin{align*}
  \NormLtwo{\jump{\hat p_0}_E}{E}^2&\leq (\NormLtwo{\hat p_0|_{T_+}}{E}+\NormLtwo{\hat p_0|_{T_-}}{E})^2\\
 &= \abs{E} (\abs{\hat T_+}\pmoh\NormLtwo{\hat p_0}{T_+}+\abs{\hat T_-}\pmoh\NormLtwo{\hat p_0}{T_-})^2\\
&\leq \abs{E} (\abs{\hat T_+}\inv + \abs{\hat T_-}\inv) \NormLtwo{\hat p_0}{\hat\omega_E}^2.
\end{align*}
The estimates $\abs{\hat T_+}\poh + \abs{\hat T_-}\poh\leq \abs{\hat T_-}\poh(1+\cquot\poh)$ and $\abs{\hat T_\pm}\pmoh \leq 2\cot\poh(\omz) \abs{E}\inv$ show
\begin{align*}
 (\abs{\hat T_+}\poh + \abs{\hat T_-}\poh)&\abs{E} (\abs{\hat T_-}\inv + \abs{\hat T_+}\inv) \\
&\leq \abs{E}(1+\cquot\poh)(\abs{\hat T_-}\pmoh+\abs{\hat T_+}\inv \abs{\hat T_-}\poh)\\
&\leq 2\cot\poh(\omz) (1+\cquot\poh)(1+\abs{\hat T_+}\pmoh\abs{\hat T_-}\poh)\\
&\leq 2\cot\poh(\omz) (1+\cquot\poh)^2 =:\csr.
\end{align*}

The estimates $\abs{\hat T_\pm}\inv \leq 4\cot(\omz) \abs{E}^{-2}$, $\abs{\hat T_\pm}\leq \abs{E}h_{\hat T_\pm} /2$, and $h_{\hat T_\pm}\leq \abs{E}/\sin(\omz)$ imply $\cquot\leq 2\cot(\omz)/\sin(\omz)$.

The summation over $\Tcal\cap\Tcalf$ and the finite overlap of $(\hat\omega_E)_{E\in\Ecalf}$ leads to
\begin{align*}
  \abs{\eta_C(\Tcal,\Tcal\cap \Tcalf)-\eta_C(\Tcalf,\Tcal\cap\Tcalf)}^2&\leq \csr \sum_{E\in\Ecalf} \NormLtwo{\nablanc(u_C-\hat u_C)}{\hat\omega_E}^2\\
  &\leq 3\csr \NormLtwo{\nablanc(u_C-\hat u_C)}{\Omega}^2.
 \end{align*}
\end{proof}

\begin{theorem}[Discrete reliability (A3) for CFEM]
 The constant $\Lambda_3=4\cot(\omz)(\kappa^2+\cjc^2) (1+6\cot(\omz)\poh(1+\cinv)) $ satisfies \eqref{eqn:dRel}.
\end{theorem}
\begin{proof}
 With solution $u_C\in\Soz$ (resp. $\hat u_C\in\Sozf$) to the discrete problem with respect to $\Tcal\in\Admis$ (resp. $\Tcalf\in\Admis(\Tcal)$), define $\hat e_C := \hat u_C-u_C$ and discrete quasi-interpolation $e_C\in \So$ of $\hat e_C\in \Sof$ from Remark \ref{rem:dQI}. The Galerkin orthogonality $a(\hat e_C,e_C)=0$, $\hat e_C - e_C=0$ on $\Tcal\cap\Tcalf$ and a piecewise integration by parts show
\begin{align*}
 \delta_C^2(\Tcal,\Tcalf)&= a(\hat u_C, \hat e_C - e_C) - a(u_C, \hat e_C - e_C) \\
&= \int_{\Tcal\setminus\Tcalf} (h_\Tcal f)h_\Tcal\inv(\hat e_C - e_C) \dx \\
&\quad\quad- \sum_{E\in\Ecal(\Om)\cap\Ecal(\Tcal\setminus\Tcal)} \int_E \jump{\nabla u_C\cdot\nu_E} (\hat e_C - e_C)\ds.
\end{align*}
The Cauchy and the trace inequality \eqref{eqn:traceInequality} prove
\begin{equation*}
 \delta_C^2(\Tcal,\Tcalf) \leq (\NormLtwo{h_\Tcal f}{\Tcal\setminus\Tcalf} + \sqrt{3}\ctr \sqrt{\sum_{E\in\Ecal(\Tcal\setminus\Tcalf)} \abs{E} \NormLtwo{\jump{\nabla u_C\cdot \nu_E}_E}{E}^2})\NormLtwoOm{h_\Tcal\inv(\hat e_C - e_C)}.
\end{equation*}
The estimates $h_K^2\leq 4\cot(\omz)\abs{K}$, $\abs{E}\leq 2\cot(\omz)\poh \abs{K}\poh$ for any $K\in\Tcal$ and the first-order approximation property \eqref{eqn:dQIApprox} prove the assertion with $\Lambda_3 = (\kappa^2+\cjc^2)(4\cot(\omz) + 6\ctr^2 \cot(\omz)\poh)$.
\end{proof}

\begin{example}
For right isosceles triangles, $\Lambda_1^2\leq 40.36$, $\Lambda_3\leq 9201 $ and \eqref{eqn:theta} lead to $\theta_0\geq 2.6 \times 10^{-6}$ for the Courant FEM, despite the general wisdom that \(\theta=0.3\) leads to optimal convergence.
\end{example}

The remaining part of this section proves an explicit bound for the bulk parameter for the Crouzeix-Raviart FEM with solution $u\crr\in\CR$ to \(a\nc(u\crr,v\crr)=(f,v\crr)_{L^2(\Om)}\) for any \(v\crr\in\CR\) with $a\nc(v\crr,w\crr):= \allowbreak (\nablanc v\crr,\nablanc w\crr)_{L^2(\Om)}$. For any admissible triangulation $\Tcal\in\Admis$ and $K\in\Tcal$, define
\begin{equation*}
 \eta\crr^2(\Tcal,K):= \abs{K} \NormLtwoK{f}^2 + \abs{K}\poh \sum_{E\in\Ecal(K)} \NormLtwo{\jump{\dell u\crr/\dell s}}{E}^2.
\end{equation*}
For $\Tcal\in\Admis$ and refinement $\Tcalf$ with solutions $u\crr\in\CR$ and $\hat u\crr\in\CRf$, define
\begin{equation*}
 \delta\crr(\Tcal,\Tcalf):= \NormEnergyncOnly{u\crr-\hat u\crr}.
\end{equation*}

The proof of stability (A1) from Theorem \ref{thm:stabCR} applies verbatim with $\dell/\dell\nu_E$ replaced by $\tau_E$ in $\dell/\dell s$.
\begin{theorem}[Stability (A1) for CRFEM]\label{thm:stabCR}
 The constants $\cquot$ from Theorem~\ref{thm:stabC} and $\Lambda_1^2= 48\cot(\omz)(2\sin(\omz))\pmoh$ satisfy \eqref{eqn:stab}.
\end{theorem}
\begin{theorem}[Discrete reliability (A3) for CRFEM]
 For a simply-connected domain $\Omega\subset\R^2$, the constant $\Lambda_3= 12\cot(\omz) (\kappa^2+\cjc^2)(1+\cinv)$ satisfies \eqref{eqn:dRel}.
\end{theorem}
\begin{proof}
 Given the solution $u\crr\in\CR$ (resp. $\hat u\crr\in\CRf$) to the discrete problem with respect to $\Tcal\in\Admis$ (resp. $\Tcalf\in\Admis(\Tcal)$), consider a discrete Helmholtz decomposition of $\nablanc u\crr\in P_0(\Tcal;\R^2)\ins P_0(\Tcalf;\R^2)$,
\begin{equation}\label{eqn:HHsplit}
 \nablanc u\crr =\nablanc \hat \alpha\crr +\Curl\hat\beta_C
\end{equation}
for unique $\hat\alpha\crr\in\CRf$ and $\hat\beta_C\in\Sof/\R$ so that
\begin{equation}\label{eqn:splitDelta}
 \delta\crr^2(\Tcal,\Tcalf)= \NormEnergyncOnly{u\crr-\hat u\crr}^2 = \NormEnergyncOnly{\hat\alpha\crr-\hat u\crr}^2 + \NormEnergy{\hat\beta_C}^2.
\end{equation}
Abbreviate $\hat v\crr:=\hat u\crr-\hat\alpha\crr\in\CRf$ and $v\crr:=\Inc\hat v\crr\in\CR$.
An analogeous proof to interpolation estimate for $\Inc:H^1_0(\Om)\rightarrow H^1_0(\Om)$ \cite[Theorem 2.1]{carstensenGedicke} with the discrete Poincar\'e constant $\cp=\sqrt{3/8}$ from Theorem \ref{thm:discPoin} and the discrete trace identity (Lemma \ref{lem:discreteTraceIdentity}) yields $\kappa\crr:=(1/8+\cp^{2})\poh=2\pmoh$ with
\begin{equation}\label{eqn:IntPolEst}
 \NormLtwoOm{h_\Tcal\inv(\hat v\crr - v\crr)}\leq \kappa\crr \NormEnergyncOnly{\hat v\crr}.
\end{equation}
Since $\hat u\crr$ solves the discrete problem on $\Tcalf$,
\begin{equation*}
 \NormEnergyncOnly{\hat u\crr-\hat\alpha\crr}^2 = a\nc(\hat u\crr,\hat v\crr)-a\nc(\hat\alpha\crr,\hat v\crr) = F(\hat v\crr) -a\nc(\hat\alpha\crr,\hat v\crr).
\end{equation*}
The orthogonal decomposition \eqref{eqn:HHsplit} and $\Pi_0\nablanc\hat v\crr = \nablanc \Inc\hat v\crr = \nablanc v\crr$ imply
\begin{align*}
 a\nc(\hat\alpha\crr,\hat v\crr) = (\nablanc u\crr,\nablanc \hat v\crr) = (\nablanc u\crr,\nablanc v\crr) = F(v\crr).
\end{align*}
The three last displayed formulas, the Cauchy inequality and $\hat v\crr-v\crr=0$ on $\Tcal\cap\Tcalf$ yield
\begin{align*}
 \NormEnergyncOnly{\hat u\crr-\hat\alpha\crr}^2 &= F(\hat v\crr-v\crr) = (f,\hat v\crr-v\crr)_{L^2(\Tcal\setminus\Tcalf)}\\
&\leq \kappa\crr \NormLtwo{h_\Tcal f}{\Tcal\setminus\Tcalf}  \NormEnergyncOnly{\hat u\crr-\hat\alpha\crr}.
\end{align*}
This and $h_K^2\leq 4\cot(\omz)\abs{K}$ for $K\in\Tcal$ show
\begin{equation}\label{eqn:EstAlphaU}
 2\NormEnergyncOnly{\hat u\crr-\hat\alpha\crr}^2\leq \NormLtwo{h_\Tcal f}{\Tcal\setminus\Tcalf}^2 \leq 4\cot(\omz) \sum_{K\in\Tcal\setminus\Tcalf} \abs{K} \NormLtwoK{f}^2.
\end{equation}

The estimate of $\NormEnergy{\hat\beta_C}$ utilizes the discrete quasi-interpolation $\beta_C\in \So$ of $\hat\beta_C\in \Sof$ from Remark \ref{rem:dQI}. A piecewise integration by parts, $\hat\beta_C=\beta_C$ on $\Tcal\cap\Tcalf$, and $\Ecal(\Tcal\setminus\Tcalf):=\bigcup_{K\in\Tcal\setminus\Tcalf} \Ecal(K)$ shows
\begin{align*}
 \NormEnergy{\hat\beta_C}^2 &= \int_\Om\Curl\hat\beta_C \cdot\nablanc u\crr\dx= \int_\Om\Curl(\hat\beta_C-\beta_C)\cdot\nablanc u\crr\dx\\
&= \sum_{K\in\Tcal\setminus\Tcalf} \int_K\Curl(\hat\beta_C-\beta_C) \cdot\nablanc u\crr\dx= \sum_{K\in\Tcal\setminus\Tcalf} \int_{\partial K} (\hat\beta_C-\beta_C)\partial u\crr/\partial s \ds\\
&= \sum_{E\in\Ecal(\Tcal\setminus\Tcalf)} \int_{E} (\hat\beta_C-\beta_C)\jump{\partial u\crr/\partial s}_E \ds.
\end{align*}
The trace identity on any $T\in\Tcal$ and $E\in\Ecal(T)$ with $v:=(\hat\beta_C-\beta_C)^2$ and the Cauchy inequality lead to
\begin{equation*}
 \abs{E}\inv \NormLtwo{\hat\beta_C-\beta_C}{E}^2 \leq \abs{T}\inv \big(\NormLtwoT{\hat\beta_C-\beta_C}^2 + h_T \NormLtwoT{\hat\beta_C-\beta_C} \NormEnergync{\hat\beta_C-\beta_C}{T}\big).
\end{equation*}
The estimate $\abs{T}\inv\leq 4\cot(\omz) h_T^{-2}$ and the weighted Young inequality for any $\lambda>0$ show
\begin{equation*}
 \abs{E}\inv \NormLtwo{\hat\beta_C-\beta_C}{E}^2 \leq 4\cot(\omz) \big((1+(2\lambda)\inv) \NormLtwoT{h_\Tcal\inv (\hat\beta_C-\beta_C)}^2 + \lambda/2 \NormEnergync{\hat\beta_C-\beta_C}{T}^2\big).
\end{equation*}
Hence, the inverse estimate and the direct minimization $\min_{\lambda>0}((2\lambda)\inv+\cinv^2\lambda/2) = \cinv$ prove, for $\ctr^2:= 4\cot(\omz)(1+\cinv)$, the trace inequality 
\begin{equation}\label{eqn:traceInequality}
 \abs{E}\inv \NormLtwo{\hat\beta_C-\beta_C}{E}^2 \leq \ctr^2 \NormLtwo{h_\Tcal\inv(\hat\beta_C-\beta_C)}{\omega_E}^2.
\end{equation}
This and the Cauchy inequality imply
\begin{align*}
 \NormEnergy{\hat\beta_C}^2 &\leq \sum_{E\in\Ecal(\Tcal\setminus\Tcalf)} \int_{E} \abs{E}\pmoh\abs{\hat\beta_C-\beta_C}\abs{E}\poh \abs{\jump{\partial u\crr/\partial s}_E} \ds\\
&\leq \sqrt{\sum_{E\in\Ecal(\Tcal\setminus\Tcalf)} \abs{E}\inv \NormLtwo{\hat\beta_C-\beta_C}{E}^2}\sqrt{\sum_{E\in\Ecal(\Tcal\setminus\Tcalf)} \abs{E} \NormLtwo{\jump{\partial u\crr/\partial s}_E}{E}^2}\\
&\leq \sqrt{3}\ctr \NormLtwoOm{h_\Tcal\inv (\hat\beta_C-\beta_C)} \sqrt{\sum_{E\in\Ecal(\Tcal\setminus\Tcalf)} \abs{E} \NormLtwo{\jump{\partial u\crr/\partial s}_E}{E}^2}.
\end{align*}
The first-order approximation property \eqref{eqn:dQIApprox} of the discrete quasi-interpolation, $\abs{E}\leq 2\cot(\omz)\poh \abs{T}\poh$, \eqref{eqn:splitDelta} and \eqref{eqn:EstAlphaU} with $2\cot(\omz)\leq 24\cot(\omz)^{3/2} (\kappa^2+\cjc^2)(1+\cinv)$ conclude the proof.
\end{proof}

\begin{example}
 For right isosceles triangles, it holds $\Lambda_1^2\leq 34.97$ and $\Lambda_3\leq 4521$ and \eqref{eqn:theta} leads to $\theta_0\geq 6.3 \times 10^{-6}$ for the Crouzeix-Raviart FEM, despite the general wisdom that \(\theta=0.3\) leads to optimal convergence.
\end{example}

\begin{acknow} 
 The authors acknowledge support of the Deutsche
      Forschungsgemeinschaft in the Priority Program 1748 \enquote{Reliable
      simulation techniques in solid mechanics. Development of non-standard discretization methods, mechanical and mathematical
      analysis} under the project \enquote{Foundation and application of
      generalized mixed FEM towards nonlinear problems in solid mechanics}.
Parts of the manuscript have been finalized while the first author enjoyed the 
fruitful atmosphere of the IHP quarter on Numerical Methods for PDEs
in Paris; the support through the program is thankfully acknowledged. 
The second author is supported by the Berlin Mathematical School.
\end{acknow}

\printbibliography

\end{document}